\documentclass[a4paper]{amsart}

\usepackage{amsmath}
\usepackage{amsthm}
\usepackage{amsfonts}
\usepackage{amssymb}
\usepackage{enumitem}
\usepackage{tikz}
\usepackage{tikz-cd}
\usepackage[colorlinks=true,linkcolor=blue,citecolor={blue}]{hyperref}
\usepackage{hyperref}
\usepackage[nocompress]{cite}
\usepackage{mathrsfs}
\usepackage{subcaption}
\usepackage[a4paper,margin=1in,headsep=25pt,footskip=35pt]{geometry}
\usepackage{todonotes}

\newcommand{\indicator}[1]{\ensuremath{\mathbf{1}_{\{#1\}}}}

\newcommand{\rfc}[2]{\boxplus_{#1}^{#2}}

\DeclareMathOperator{\He}{He}

\DeclareMathOperator{\dd}{d\!}

\newcommand{\E}{\mathbb{E}}
\newcommand{\C}{\mathbb{C}}

\renewcommand\Re{\operatorname{Re}}
\renewcommand\Im{\operatorname{Im}}
\newcommand{\eps}{\varepsilon}

\newcommand{\R}{\mathbb{R}}
\newcommand{\N}{\mathbb{N}}

\def\R{\mathbb{R}}

\theoremstyle{plain}
\newtheorem{theorem}{Theorem}[section]
\newtheorem{conjecture}[theorem]{Conjecture}
\newtheorem{lemma}[theorem]{Lemma}

\newtheorem{proposition}[theorem]{Proposition}

\theoremstyle{definition}
\newtheorem{definition}[theorem]{Definition}

\newtheorem{assumption}[theorem]{Assumption}

\newtheorem{question}[theorem]{Question}
\newtheorem{remark}[theorem]{Remark}

\numberwithin{equation}{section}

\begin{document}
	
	\title[Appell polynomials and finite free probability]{Free infinite divisibility, fractional convolution powers, and {A}ppell polynomials}

		\author{Andrew Campbell}
		\address{Institute of Science and Technology Austria, Am Campus 1, 3400 Klosterneuburg, Austria}
		\email{andrew.campbell@ist.ac.at}

	\begin{abstract}
Initiated by  a result of Gorin and Marcus (2020)  and  an observation of Steinerberger (2019), there has been a recent growing body of literature connecting repeated differentiation of real rooted polynomials to free additive convolution semigroups in free probability. Roughly, this connection states that in the large degree limit the empirical measure of the roots after many derivatives is, up to a rescaling, the original empirical measure of the roots raised to a free additive convolution power. If the original roots satisfy some bounds and the number of derivatives is such that the remaining degree is fixed, then it  has been shown in various contexts that these high derivatives converge to the Hermite polynomials. In the context of convolution semigroups and finite free probability, where Hermite polynomials are the analogue of the Gaussian distribution, these results have a natural interpretation as a central limit theorem for repeated differentiation.
		
		We consider the case when these root bounds are removed and identify the potential limits of repeated differentiation as the real rooted Appell sequences. We prove that a sequence of polynomials is in the domain of attraction of an Appell sequence exactly when the empirical measures of the roots are, up to a rescaling, in the domain of attraction of a free infinitely divisible distribution naturally associated to the Appell sequence. We consider the limits of Appell sequences, generalizing the well known fact that the roots of a Hermite polynomial, after being appropriately normalized, are asymptotically distributed according to the semicircle distribution. We additionally extend these notions of infinite divisibility and fractional convolution semigroups to rectangular finite free probability.  
		
		Our approach is based on the finite free $R$-transform of the polynomials, providing a step towards an analytic theory of finite free probability. These transforms provide a clear connection between Appell polynomials and free infinitely divisible distributions, where the finite free $R$-transform of a real rooted Appell polynomial is a truncated version of the $R$-transform of an infinitely divisible distribution.
	\end{abstract}

	\maketitle 
	
		\section{Introduction}\label{sec:intro} As is well known, the empirical distribution of the normalized roots of the $d$-th Hermite polynomial converges, as $d\rightarrow\infty$, to the semicircle law with density $\frac{\sqrt{4-x^2}}{2\pi}$ on $[-2,2]$. In Voiculescu's  \emph{free probability theory} \cite{Voiculescu1986}, the semicircle distribution  is the analogue of the Gaussian distribution, and the Hermite polynomials serve as the Gaussian distribution of the recently formulated theory of \emph{finite free probability} of Marcus, Spielman, and Srivastava \cite{Marcus-Spielman-Srivastava2022}. We consider one generalization of this connection, motivated by a question of Steinerberger \cite{Steinerberger2019} on the dynamics of polynomial roots under the action of a differential operator.  The random matrix connection to Steinerberger's question was in fact considered in the earlier work of Gorin and Marcus \cite{Gorin-Marcus2020}, but we start from the perspective of this as a question outside the realm of random matrix theory as a motivation for potential future generalizations. We connect the real rooted \emph{Appell sequences}, i.e.\ sequences of polynomials $\{A_{d}\}_{d=1}^{\infty}$ indexed by their degree with only real roots satisfying \begin{equation}\label{eq:Appell def}
		d A_{d-1}(x)= A_{d}'(x),
	\end{equation} to free infinitely divisible (or \emph{ID}) distributions. Specifically:\begin{enumerate}
		\item In Theorem \ref{thm:Appell and free ID}, we define a normalization such that the root distributions of an Appell sequence converge to a  free ID distribution. This is the same normalization needed to recover the semicircle law from the roots of Hermite polynomials.
		
		\item In Theorem \ref{thm:Appell domain}, we provide the domain of attraction for an Appell sequence under repeated differentiation, i.e.\  conditions on a sequence of real rooted polynomials $\{p_{d}\}_{d=1}^\infty$ indexed by their degree such that for any $\ell\in \N$ \begin{equation*}
			\lim\limits_{d\rightarrow\infty} \frac{\ell!}{d!}\left(\frac{\dd}{\dd x} \right)^{d-\ell}p_d(x)=A_{\ell}(x),
		\end{equation*} uniformly on compact subsets. These conditions are such that the empirical root distribution of $p_{d}$, scaled by $\frac{1}{d}$, is in the domain of attraction of the free ID distribution of Theorem \ref{thm:Appell and free ID}.
	\end{enumerate} We also pose a conjecture for recovering free stable distributions from a different normalization of the roots of Appell sequences. We extend these notions to the \emph{rectangular (finite) free convolution} and rectangular free ID distributions, where repeated differentiation is replaced by another differential operator naturally associated to rectangular finite free probability. In doing so we additionally prove results on fractional rectangular free convolutions powers analogous to recent results on repeated differentiation  and free additive convolution powers.
	
	Central to our proofs is the \emph{finite free $R$-transform} of these polynomials. This $R$-transform approach provides an analytic counterpart to cumulant based approaches which are popular in finite free probability. The $R$-transform also allows us to see the connection between infinite divisibility and the Appell sequences by  connecting the sequence to a function $f$ in the \emph{Laguerre--P\'olya} class, such that the limiting $R$-transform is of the form $R(z)=-\frac{f'(z)}{f(z)}$. Using this $R$-transform approach additionally allows us to prove results for the square and rectangular finite free additive convolutions in parallel, with essentially only notational changes between proofs. 
	
	In the remainder of the introduction we discuss the problem consider by Steinerberger \cite{Steinerberger2019}, provide some background in (finite) free probability, and give a brief introduction to Appell polynomials and functions in the Laguerre--P\'olya class. Section \ref{sec:main results} then contains our main results. This includes various connections between Appell polynomials and infinite divisibility in Section \ref{sec:ID and Appell}, results on rectangular free convolution powers and polynomials under a certain differential operator in Section \ref{sec:fract rect conv}, and results on rectangular free ID distributions are provided in Section \ref{sec:rect ID and L_Appell}. We then discuss future directions and implications of our results in Section \ref{sec:further directions}, including an extended discussion of connections to the \emph{$\beta$-corners} process in random matrix theory. We then conclude with the proofs in Section \ref{sec:proofs}.
	
	\subsection{The limits of repeated differentiation} Our investigation is motivated by the following question.
	\begin{question}\label{ques:Root question}
		Let $\mathcal{L}:\C[z]\rightarrow\C[z]$ be a linear operator on polynomials. Given some information on the roots of a polynomial $p$, what can be said about the roots of $\mathcal{L}^{N}p$ for some $N\in\N$?
	\end{question}
	The exact interpretation of Question \ref{ques:Root question} will depend on the goals of whoever is attempting to answer it. A fundamental question in the geometry of polynomial roots is whether $\mathcal{L}$ preserves the property of all roots belonging to a certain domain (such as the real line), see for example the work of Borcea and Br\"ad\'en \cite{Borcea-Branden2009} which (among other things) completely characterized operators preserving real rooted-ness. We consider the perspective of tracking the evolution of roots as $N$ increases for particular choices of $\mathcal{L}$, which Steinerberger \cite{Steinerberger2019} considered for the differentiation operator. 
	
	As was observed by Steinerberger \cite{Steinerberger2019,Steinerberger2020} repeatedly differentiating a real rooted polynomial has the same effect in the large degree limit on the root density as taking the \emph{fractional free convolution power} of the original root density. Steinerberger's original paper \cite{Steinerberger2019} was largely outside the range of random matrix theory, and this observation was made by comparing the PDEs studied in \cite{Steinerberger2019,Shlaykhtenko-Tao2020} describing the evolution under these processes.  It should be noted however, that repeated differentiation had been considered earlier in a paper of Gorin and Marcus \cite{Gorin-Marcus2020} as the $\beta\rightarrow\infty$ limit of the \emph{$\beta$-corners process} in random matrix theory. This earlier result already suggested random matrix limits of the operation on polynomials.  This has been made rigorous and extended by multiple authors \cite{Arizmendi-GarzaVargas-Perales2023,Hoskins-Kabluchko2021,Campbell-ORourke-Renfrew2024,Campbell-ORourke-Renfrew2024even,Hall-Ho-Jalowy-Kabluchko2023Repeat,Arizmendi-Fujie-Perales-Ueda2024,ORourke-Steinerberger2021,Gorin-Klepttsyn2020universal}. In particular \cite{Arizmendi-GarzaVargas-Perales2023,Arizmendi-Fujie-Perales-Ueda2024,Campbell-ORourke-Renfrew2024even} place Steinerberger's observation in the context of \emph{finite free probability} to show that in an appropriate sense this connection holds on all possible scales. The available proofs tend to follow a ``moment method'' approach, by computing the (finite free) cumulants of the empirical measure of the roots for some high derivative of the polynomial.	
	
	To be slightly more precise, for a degree $d$ polynomial $p$ with real roots $x_{1},\dots,x_{d}$ we consider the empirical root measure (ERM) of the polynomial $\mu_{p}=\frac{1}{d}\sum_{j=1}^{d}\delta_{x_j}$, where $\delta_{x}$ is a Dirac measure at $x$. Our guiding heuristic for repeated differentiation is that if $p_{j}(x)=p^{(j)}\left(\frac{d-j}{d}x\right)$ is the rescaled $j$-th derivative of $p$, then \begin{equation}\label{eq:Heuristic}
		\mu_{p_j}\approx \mu_{p}^{\boxplus\frac{d}{d-j}},
	\end{equation} where $\boxplus$ is the \emph{free additive convolution}. In \cite{Hoskins-Kabluchko2021,Arizmendi-GarzaVargas-Perales2023} \eqref{eq:Heuristic} was first made rigorous when $j\sim td$ for $t\in(0,1)$ in the large degree limit using techniques from free probability and finite free probability. In finite free probability, the degree $d$ finite free additive convolution $\boxplus_{d}$ is binary operation on polynomials preserving real rooted-ness, which converges as $d\rightarrow\infty$ to $\boxplus$ in the sense of ERMs. In \cite{Arizmendi-Fujie-Perales-Ueda2024} \eqref{eq:Heuristic} was given an interpretation directly on the level of polynomials. Though $\boxplus_{d}$ depends on $d$, and the degrees of $p_j$ and $p$ do not align, the authors of \cite{Arizmendi-Fujie-Perales-Ueda2024} observed that the degree $d-j$ \emph{finite free cumulants} of $p_{j}$ are exactly $\frac{d}{d-j}$ times the degree $d$ finite free cumulants of $p$, which is what you would expect for the $\frac{d}{d-j}$ convolution power. 

	Given \eqref{eq:Heuristic} it is natural to ask what happens when $j\approx d$. This question should be broken up into two distinct regimes, depending on whether $\ell=d-j=o(d)$ remains fixed or tends to $\infty$. In the latter case the answer should be formulated in terms of convergence of ERMs in the large degree limit. While explicit statements of this case are absent from the literature, \cite[Proposition 3.4]{Arizmendi-Fujie-Perales-Ueda2024} should be sufficient to resolve it for polynomial sequences with uniformly bounded roots with little additional effort. 
	
	When $\ell$ is fixed, the answer should be formatted in terms of convergence of polynomials. Hoskins and Steinerberger \cite{Hoskins-Steinerberger2022} proved that if $p$ is a random polynomials with independent and identically distributed (iid) roots with all moments finite, then after an appropriate centering and rescaling $p_{d-\ell}$ converges to the degree $\ell$ Hermite polynomial \begin{equation}\label{eq:Hermite def}
		\He_\ell(x)=\sum_{k=0}^{\lfloor \frac{\ell}{2}\rfloor}\frac{\ell!(-1)^{k}}{k!(\ell-2k)!}\frac{x^{\ell-2k}}{2^{k}}.
	\end{equation} Hermite polynomials also appeared as the universal limit of repeated differentiation in a seemingly very different context in the work of Griffin, Ono, Rolen, and Zagier \cite{Griffin-Ono-Rolen-Zagier2019}.  Gorin and Kleptsyn \cite{Gorin-Klepttsyn2020universal} proved convergence to Hermite polynomials for general deterministic sequences of polynomials with certain root bounds without using finite free probability.  The author, O'Rourke, and Renfrew \cite{Campbell-ORourke-Renfrew2024even} used finite free probability to consider universal limits of repeated differentiation, where the limit is either $(x-1)^{\ell}$ or $\He_{\ell}(x)$ depending on whether $p_j$ is scaled analogously to the law of large numbers or central limit theorem, respectively. In the context of finite free probability these limits are natural, as they represent the finite free versions of the Dirac and Gaussian distributions. 
	
	However, the existence of a few very large roots prevents convergence to the Hermite polynomials. This is similar to random variables outside the domain of attraction of the Gaussian distribution in the classical central limit theorem, where among $d$ iid samples $1$ will contribute disproportionally to the sum. It is then natural to ask what families of real rooted polynomials $\{A_{\ell}\}_{\ell=1}^{\infty}$ can arise as such limits, as an analogue of infinitely divisible distributions generalizing the classical Gaussian central limit theorem. However, Arizmendi and Perales \cite{Arizmendi-Perales2018} proved that $(x-1)^{\ell}$ and $\He_{\ell}$ are the only $\boxplus_{\ell}$-ID distributions. If we assume the polynomials are monic, then clearly they must satisfy \eqref{eq:Appell def} and are thus an Appell sequence. 
	
	We will associate to every Appell sequence a specific $\boxplus$-ID distribution and we show in Theorem \ref{thm:Appell and free ID} how to recover this distribution from the Appell sequence itself in a way analogous to that of Hermite polynomials and the semicircle law. Key to this connection is encoding the real rooted Appell sequence into a function in the \emph{Laguerre--P\'olya class}. 

	In Theorem \ref{thm:Limit theorem} we extend the heuristic \eqref{eq:Heuristic} to the operator $xD^2+(n+1)D$, where $D=\frac{\dd}{\dd x}$ and the free additive convolution is replaced by the \emph{rectangular} free additive convolution $\boxplus_{\lambda}$. In turn, we then prove in Theorem \ref{thm:L-A and Rect ID} the convergence of Appell like sequences to certain $\boxplus_{\lambda}$-ID distribution.

	\subsection{Notation} We denote by $\C^{+}=\{z\in\C: \Im(z)>0 \}$ the upper half complex plane and by $\C^{-}=\{z\in\C: \Im(z)<0\}$ the lower half complex plane.
	
	We employ the asymptotic notation $O(\cdot)$, $o(\cdot)$, $\lesssim$, $\sim$, etc.\ under the assumption that some sequence index, such as $d$ or $\ell$, tends to infinity. We write $c_d=O(b_d)$ if there exists some constant $C>C'>0$ such that $|c_d|\leq C|b_d|$ for all $d > C$, $c_d=o(b_d)$ if $\frac{c_d}{b_d}\rightarrow0$, $c_{d}=\Theta(b_{d})$ if $C'b_{d}<c_{d} <Cb_{d}$, and $c_d\sim b_d$ if $\frac{c_d}{b_d}\rightarrow1$.   
	
	For polynomials $p$ and $q$ we use the notation $p\approx q$ if $p(x)=cq(x)$ for some constant $c\neq 0$. We define for any $\alpha>0$ the dilation operator $\mathcal{D}_{\alpha}$ on degree $d$ polynomials as $\mathcal{D}_{\alpha}p(x)=\alpha^{d}p(x/\alpha)$. This operator multiples the roots by $\alpha$.
	
	For a real number $d$ and natural number $k$ we shall use $(d)_k=d(d-1)\cdots(d-k+1)$ to denote the falling factorial. A partition, $\pi = \{ V_1, \ldots, V_r \}$ of $[j] := \{ 1,\ldots, j\}$ is a collection of pairwise disjoint, non-null, sets $V_i$ such that $\cup_{i=1}^{r} V_i = [j]$. We refer to $V_i$ as the blocks of $\pi$, and denote the number of blocks of $\pi$ as $|\pi|$. The set of all partitions of $[j]$ is denoted $\mathcal{P}(j)$.
	
	The set of partitions can be equipped with the partial order $\preceq$ of reverse refinement, where we define $\pi \preceq \sigma$ if every block of $\pi$ is completely contained in a block of $\sigma$. The minimal element in this ordering is $0_j : = \{ \{1\}, \{2\}, \ldots, \{j\} \}$ and the maximal element is $1_j = \{\{1,2, \ldots, j \}\}$. The supremum of $\pi$ and $\sigma$ is denoted $\pi \vee \sigma$.  For a partition $\pi =  \{ V_1, \ldots, V_r \}$ and a sequence of numbers $\{c_n\}$, we use 
	\begin{equation} \label{eq:cpinotation}
		c_{\pi}:= \prod_{i=1}^{|\pi|}  c_{|V_i|}.
	\end{equation} Additionally, we use $N!_{\pi}$ to denote $\prod_{V\in\pi}|V|!$. 
	
	\subsection{Definitions in (finite) free probability} In this section we collect some definitions needed to state our main results. For a more extensive background on free probability see for example \cite{Mingo-Speicher2017,Nica-Speicher2006}. 
	
	The free additive convolution $\boxplus$ is a binary operation on probability measures supported on the real line which serves as the free analogue of the classical convolution $*$. To define $\boxplus$ for general probability measures, we follow the notation of \cite[Definition 29]{Mingo-Speicher2017} to first define the $R$-transform of a probability measure. For a probability measure $\nu$ on $\R$ the Cauchy transform $G_{\nu}$ of $\nu$ is the analytic function from $\C^{+}$ to $\C^{-}$ such that \begin{equation*}
		G_{\nu}(z)=\int_{\R}\frac{1}{z-t}\dd\nu(t).
	\end{equation*} Let $\alpha,\beta>0$, $\Gamma_{\alpha,\beta}=\{z:\alpha\Im(z)>|\Re(z)|,\ \Im(z)>\beta \}$, and $\Delta_{\alpha,\beta}=\{z:z^{-1}\in\Gamma_{\alpha,\beta}\}$.
	\begin{definition}
		Let $\nu$ be a probability measure on $\R$ and $G_{\nu}$ its Cauchy transform. For every $\alpha>0$ there exists $\beta>0$ and a germ $R_{\nu}$ of analytic functions on $\Delta_{\alpha,\beta}$ satisfying \begin{equation}
			G_\nu\left(R_{\nu}(z)+\frac{1}{z}\right)=z,\text{ for all }z\in\Delta_{\alpha,\beta},
		\end{equation}  and \begin{equation*}
		R_{\nu}\left(G_{\nu}(z)\right)+\frac{1}{G_{\nu}(z)}=z,\text{ for all }z\in\Gamma_{\alpha,\beta}.
	\end{equation*} We define $R_{\nu}$ to be the $R$-transform of $\nu$. 
	\end{definition} \begin{definition}
	For probability measures $\nu$ and $\mu$ on $\R$, the free additive convolution $\nu\boxplus\mu$ is the unique probability measure such that $R_{\nu\boxplus\mu}=R_{\nu}+R_{\mu}$. 
\end{definition}

	Just as in classical probability, there exists a notion of free infinite divisibility, or $\boxplus$-infinite divisibility. These notions in fact turn out to be explicitly in bijection with one another, and this bijection has been dubbed the \emph{Bercovici--Pata bijection} \cite{Bercovici-Pata1999}. We say a probability measure $\nu$ is $\boxplus$-infinitely divisible (or $\boxplus$-ID) if for any $n\in\N$, there exists a probability measure $\nu_n$ such that $\nu_{n}^{\boxplus n}=\nu$. We have the following explicit characterization of $\boxplus$-ID distributions which serves as the free version of the L\'evy--Khintchine formula for $*$-ID distributions. \begin{theorem}[See \cite{Bercovici-Voiculescu1993}]\label{rthm:ID rep}
		Let $\nu$ be a probability measure on $\R$. $\nu$ is $\boxplus$-ID if and only if there exists a positive finite Borel measure $\sigma$ on $\R$ and a real number $\gamma$ such that \begin{equation*}
			R_{\nu}(z)=\gamma+\int_{\R}\frac{z+t}{1-tz}\dd\sigma(t).
		\end{equation*}
	\end{theorem}  An important subclass of $\boxplus$-ID distributions are the $\boxplus$-stable distributions. We say that  a non-degenerate probability measure  $\mu$ is $\boxplus$-stable with stability index $\alpha\in (0,2]$ if $\mu\boxplus \mu=\mathcal{D}_{2^{1/\alpha}}\mu$, where $\mathcal{D}_{2^{1/\alpha}}\mu$ is the dilation of $\mu$ by $2^{1/\alpha}$. Finally we conclude with the following result on non-integer $\boxplus$ powers. \begin{theorem}[See \cite{Nica-Speicher1996}]
	Let $\nu$ be a probability measure. Then, for any real $k\geq 1$ there exists a probability measure $\nu^{\boxplus k}$ such that $R_{\nu^{\boxplus k}}(z)=kR_{\nu}(z)$. 
\end{theorem}

	One place where $\boxplus$ arises naturally is in the large dimension limits of random matrices. For an $N\times N$ random matrix $A_N$, with eigenvalues $\lambda_{1},\dots,\lambda_{N}$, we define \emph{empirical spectral measure} as $\mu_{A_{N}}=\frac{1}{N}\sum_{j=1}^{N}\delta_{\lambda_{j}}$. If $\{A_N\}$ and $\{B_{N}\}$ are independent sequences of square random matrices indexed by their dimension, $\mu_{A_{N}}\rightarrow\mu_{A}$, $\mu_{B_{N}}\rightarrow\mu_{B}$ as $N\rightarrow\infty$, and the distribution of every $A_N$ is invariant under unitary conjugation, then $\mu_{A_{N}+B_{N}}\rightarrow\mu_{A}\boxplus\mu_{B}$. Benaych-Georges \cite{Benaych-Georges2007,Benaych-Georges2009,Benaych-Georges2010} established a theory of \emph{rectangular} free probability and defined an analogous convolution on symmetric probability measures which describes the limit for empirical singular value distributions of rectangular random matrices. 
	
	Let $\lambda\in[0,1]$. To define the rectangular free additive convolution $\boxplus_{\lambda}$ we first define\footnote{For brevity we omit many of the details on the invertibility of various transforms and encourage the reader to see \cite{Benaych-Georges2009}.} the rectangular $R$-transform with ratio $\lambda$. Let $\mu$ be a symmetric probability measure and define the transforms \begin{equation*}
		M_{\mu^2}(z):=\int_{\R}\frac{zt^2}{1-t^2z}\dd\mu(t),
	\end{equation*} and \begin{equation*}
	H_{\mu}(z):=z\left(\lambda M_{\mu^2}(z)+1\right)\left(M_{\mu^2}(z)+1\right).
\end{equation*} Additionally, define the analytic function in a neighborhood of zero $U(z)=\frac{-\lambda-1+\sqrt{(\lambda+1)^2+4\lambda z}}{2\lambda}$, where $\sqrt{\cdot}$ is analytic square root on $\C\setminus(-\infty,0]$. It turns out $H_{\mu}$ is invertible on $(-\beta,0)$ for some small $\beta>0$, see \cite{Benaych-Georges2009} for example, and the rectangular $R$-transform with ratio $\lambda$ can be defined by \begin{equation*}
C_{\mu}^{\lambda}(z):=U\left(\frac{z}{H_{\mu}^{-1}(z)}-1 \right).
\end{equation*} For any two symmetric probability measures $\mu$ and $\nu$ on $\R$, the rectangular free additive convolution with ratio $\lambda$ is the unique symmetric probability measure $\mu\boxplus_{\lambda}\nu$ such that $C_{\mu\boxplus_{\lambda}\nu}^{\lambda}=C_{\mu}^{\lambda}+C_{\nu}^{\lambda}$.

	In \cite{Benaych-Georges2007} a description of $\boxplus_{\lambda}$-ID distributions was given. \begin{theorem}[Theorem 2.5 of \cite{Benaych-Georges2007}]
		A symmetric probability measure $\nu$ on $\R$ is $\boxplus_{\lambda}$-ID if and only if there exists a positive finite symmetric measure $G$ on the real line such that $C_{\nu}^{\lambda}(z)=z\int_{\R}\frac{1+t^2}{1-zt^2}\dd G(z)$. 
	\end{theorem} The most important measure in rectangular free probability is the  \emph{$\lambda$ rectangular free Gaussian} with mean zero and variance one. An instructive description of this distribution  is in terms of the \emph{Marchenko--Pastur distribution with rate $c$}, for $c\geq 1$. This is the measure with density \begin{equation*}
	f_{c}(x)=\frac{\sqrt{4c-(x-1-c)^2}}{2\pi x},
\end{equation*} on $[(1-\sqrt{c})^{2},(1+\sqrt{c})^{2} ]$. Then, the $\lambda$ rectangular free Gaussian with mean zero and variance one is the symmetric probability measure $\mu_{\lambda}$ such that the push-forward of $\mu_{\lambda}$ by $t\mapsto t^2$ is the push-forward of Marchenko--Pastur distribution with rate $\frac{1}{\lambda}$ by $t\mapsto \lambda t$. 

	For measures with all moments finite it is possibly to define \emph{free cumulants} and \emph{rectangular free cumulants} which linearize these convolutions. The $R$-transform then has a series expansion at zero with coefficients given by these free cumulants. 
	
	Finite free probability arose out the work of Marcus, Spielman, and Srivastava \cite{Marcus-Spielman-Srivastava2015-1,Marcus-Spielman-Srivastava2015-2,Marcus-Spielman-Srivastava2018,Marcus-Spielman-Srivastava2022,Marcus-Spielman-Srivastava2022Inter} on interlacing polynomials. In finite free probability one identifies a polynomial $p$ with its ERM $\mu_{p}$ and a matrix for which $p$ is the characteristic polynomial. Then, certain operations on polynomials can be described via expected values of random matrices, and hence have a free probabilistic interpretation. For degree $d$ monic polynomials $p(x)=x^{d}+\sum_{k=1}^{d}(-1)^{k}a_{k}x^{d-k}$ and $q(x)=x^{d}+\sum_{k=1}^{d}(-1)^kb_{k}x^{d-k}$ the most important such operation\footnote{To avoid confusion it is worth noting the rectangular free additive convolution $\boxplus_{\lambda}$ is a binary operation on measures while $\boxplus_{d}$ is a binary operation on polynomials.} is the \emph{finite free additive convolution} $p\boxplus_{d}q$ of $p$ and $q$ defined by \begin{equation*}
		p\boxplus_{d}q(x):=x^{d}+\sum_{k=1}^{d}(-1)^{k}\sum_{i+j=k}\frac{(d-i)!(d-j)!}{d!(d-k)!}a_{i}b_{j}.
	\end{equation*} The convolution $\boxplus_{d}$ has gained recent attention for the following alternative definition: If $A$ and $B$ are $d\times d$ real symmetric matrices such that $p(x)=\det(x-A)$ and $q(x)=\det(x-B)$, then $(p\boxplus_dq)(x)=\E_{O}\det\left(x-(A+O^*BO) \right)$ where the expected value is taken with respect to a Haar distributed orthogonal matrix $O$. A related important property is that if $p$ and $q$ have only real roots, then so does $p\boxplus_{d}q$. In fact, the study of $\boxplus_{d}$ dates back over a century to Walsh \cite{Walsh22}, though these descriptions in terms of random matrices are recent. In the large $d$ limit, the finite free additive convolution converges to the free additive convolution in the following sense.  \begin{proposition}[See \cite{Arizmendi-Perales2018}]\label{prop: finite conv to free conv}
	Let $\{p_d\}$ and $\{q_d\}$ be real rooted polynomials indexed by their degrees such that $\mu_{p_d}\rightarrow\mu_{p}$ and $\mu_{q_{d}}\rightarrow\mu_{q}$ weakly in the sense of probability measures as $d\rightarrow\infty$. Then, $\mu_{p\boxplus_{d}q}\rightarrow\mu_{p}\boxplus\mu_{q}$ as $d\rightarrow\infty$.
\end{proposition} Proposition \ref{prop: finite conv to free conv} was published for measures with compact support in \cite{Arizmendi-Perales2018}, however it appeared earlier in the unpublished work \cite{Marcus2021} and the result for any measures $\mu$ and $\nu$ first appeared in the preprint \cite{Arizmendi-Fujie-Perales-Ueda2024}.

	Arizmendi and Perales \cite{Arizmendi-Perales2018} defined \emph{finite free cumulants} $\{\kappa_{k}^{d}\}_{k=1}^{d}$ which linearize $\boxplus_{d}$. There they establish that the finite free cumulants converge to the free cumulants in the $d\rightarrow\infty$ limit. We will only discuss the relationship between finite free cumulants and the moments of a polynomial here. For a degree $d$ polynomial $p$ with roots $x_{1},\dots,x_{d}$ the $j$-th moment of $p$, denoted $m_{j}(p)$, is the $j$-th moment of its empirical root measure \begin{equation*}
		m_{j}(p)=\frac{1}{d}\sum_{k=1}^{d}x_{k}^{j}.
	\end{equation*} The degree $d$ finite free cumulants could then be defined by the moment-cumulant formula \begin{equation}\label{eq:moment cumulant formula}
	\kappa_{j}^{d}(p)=\frac{(-d)^{j-1}}{(j-1)!}\sum_{\pi\in\mathcal{P}(j)}d^{|\pi|}m_{\pi}(p)\left[\prod_{V\in\pi}(|V|-1)!\right]\sum_{\sigma\geq \pi}\frac{(-1)^{|\sigma|}(|\sigma|-1)!}{(d)_{\sigma}}.
\end{equation} The differentiation operator also has an interpretation as a compression of a matrix in finite free probability. \begin{lemma}[Lemma 1.17 in \cite{Marcus-Spielman-Srivastava2022}]\label{lem:compression lemma}
Let $\ell<d$, $A$ a $d\times d$ matrix, $p(x)=\det(x-A)$, and let $O$ be a uniformly distributed random $\ell\times d$ matrix with orthonormal rows, then \begin{equation*}
	\E_{O}\det\left(x-OAO^*\right) =\frac{\ell!}{d!}D^{d-\ell}p(x).
\end{equation*} 
\end{lemma} We can view the matrix $OAO^*$ as a compression of $A$ onto a uniformly chosen $\ell$ dimensional subspace. 
	
	It was also observed in  \cite{Marcus-Spielman-Srivastava2022} that $\boxplus_d$ can  be computed in terms of differential operators; namely if $\widehat p$ and $\widehat q$ are such that $\widehat p(D) x^d= p(x) $ and $\widehat q(D)x^d = q(x) $, where $D$ denotes the differentiation operator, then \begin{equation}\label{eq:additive fourier transform}
		p(x) \boxplus_d q(x) =  \widehat p(D)  \widehat q(D)  x^d.
	\end{equation}

	We also consider the \emph{rectangular finite free additive convolution} $\rfc{d}{n}$ defined for $n\in\N_{0}$ by \begin{equation}\label{eq:int n rfc}
		(p\rfc{d}{n}q)(x)=x^{d}+\sum_{k=1}^{d}(-1)^{k}x^{d-k}\sum_{i+j=k}\frac{(d-i)!(d-j)!}{d!(d-k)!}\frac{(n+d-i)!(n+d-j)!}{(n+d)!(n+d-k)!}a_{i}b_{j}.
	\end{equation} The convolution $\rfc{d}{n}$ has many analogous properties to that of $\boxplus_{d}$. If $A$ and $B$ are $(d+n)\times d$ matrices such that $p(x)=\det(x-A^*A)$ and $q(x)=\det(x-B^*B)$, then $(p\rfc{d}{n}q)(x)=\E_{U,O}\det\left(x-(A+UBO)^*(A+UBO) \right)$ where the expected value is taken over independent Haar distributed orthogonal matrices $U$ and $O$. A related important property is that if $p$ and $q$ have only non-negative roots, then so does $(p\rfc{d}{n}q)$.

	We also consider the differential operator \begin{equation}\label{eq:M def}
	M_{n}=xD^{2}+(n+1)D,
\end{equation} defined for any real $n>-1$. Essentially, as we will demonstrate throughout, anything that is true about $D$ and $\boxplus_{d}$ also holds for $M_n$ and $\rfc{d}{n}$. The following lemma, whose proof follows from a straightforward computation, is one such example. The second identity, \eqref{eq:rectangular gen}, appeared in \cite[Proof of Main Result III]{Cuenca24}. \begin{lemma}\label{thm:M def of rfc}
		For degree $d$ monic polynomials $p$ and $q$, and any $n\in\N_{0}$\begin{equation}\label{eq:M def of rfc}
			(p\rfc{d}{n}q)(x)=\frac{1}{d!(d+n)_{d}}\sum_{k=0}^{d}\left[M_{n}^{d-k}q \right](0)M_{n}^{k}p(x),
		\end{equation} where $M_{n}$ is defined in \eqref{eq:M def}. Additionally, if $p(x)=P\left(M_{n} \right)x^d$ and $q(x)=Q\left(M_{n} \right)x^d$  for some formal power series $P$ and $Q$, then \begin{equation}\label{eq:rectangular gen}
		(p\rfc{d}{n}q)(x)=P\left(M_{n} \right)Q\left(M_{n}\right)x^d.
	\end{equation}
\end{lemma} The operator $M_n$ also acts as the compression of rectangular matrices, in a way completely analogous to $D$ in Lemma \ref{lem:compression lemma}. We could not find it explicitly stated in the literature, however it follows from straight-forward computations using the notion of \emph{minor orthogonality} defined in \cite{Marcus-Spielman-Srivastava2022}, and may follow from steps of proofs in \cite{marcus2016discreteunitaryinvariance,Gribinski2024}. \begin{lemma}\label{lem:rect matrix compression}
Let $A$ be a $d\times (d+n)$ rectangular matrix and let $p(x)=\det(x-AA^{*})$. Let $\ell<d$. Let $U$ be a uniformly distributed $\ell\times d$ random matrix with orthonormal rows and let $O$ be a uniformly distributed $(d+n)\times (\ell+n)$ random matrix with orthonormal columns. Then, \begin{equation}\label{eq:rect comp}
	\E_{U,O}\det\left(x-[UAO][UAO]^* \right)=\frac{\ell!(\ell+n)!}{d!(d+n)!}M_n^{d-\ell}p(x).
\end{equation}
\end{lemma}

	With Lemma \ref{thm:M def of rfc} we can easily extend the definition of $\rfc{d}{n}$ to fractional values of $n$.  \begin{definition}
		Let $p$ and $q$ be monic polynomials. For any real $n>-1$, we define the rectangular convolution $p\rfc{d}{n}q$ by \eqref{eq:M def of rfc}. 
	\end{definition} To connect $\rfc{d}{n}$ to $\boxplus_{\lambda}$ we need to associate a symmetric probability measure to a polynomial $p(x)=\prod_{j=1}^{d}(x-\alpha_{j}^2)$ with non-negative roots $\alpha_{1}^{2},\dots,\alpha_{d}^{2}$. For a measure $\mu$ on the positive half line, we denote by $\sqrt{\mu}$ the symmetrized version of the push-forward of $\mu$ by the map $t\mapsto\sqrt{t}$. For the ERM of $p$ it is then straightforward to see that \begin{equation}\label{eq:sqrt of meas def}
\sqrt{\mu_{p}}=\frac{1}{2d}\sum_{j=1}^{d}\delta_{\alpha_{j}}+\delta_{-\alpha_{j}}.
\end{equation} Studying positively supported measures through their symmetrized square roots is common in random matrix theory and free probability, and is one reason one might want to extend the definition of  $\rfc{d}{n}$ to fractional $n>-1$. If $\tilde{p}$ is the degree $2d$ polynomials with roots $\pm\alpha_{1},\dots,\pm\alpha_{d}$, then $D^{2}\tilde p(x)=4[M_{-1/2}p](x^2)$.

	Cuenca \cite{Cuenca24} defined rectangular finite free cumulants which linearize $\rfc{d}{n}$. It follows from the results of Gribinski \cite{Gribinski2024}, and was explicitly pointed out and independently proven in \cite{Cuenca24}, that if $\frac{d}{d+n}\rightarrow\lambda\in (0,1]$, then $\rfc{d}{n}$ converges to $\boxplus_{\lambda}$ in a way completely analogous to Proposition \ref{prop: finite conv to free conv}. The most important polynomials in rectangular finite free probability are the \emph{Laguerre polynomials} with index $n>-1$\begin{equation}\label{eq:Laguerre polynomial def}
		L^{(n)}_{d}(x)=\sum_{k=0}^{d}\frac{(-1)^{k}(d+n)_{d-k}}{k!(d-k)!}x^{k}.
	\end{equation} $L_{d}^{(n)}$ serves the role as the $\rfc{d}{n}$ Gaussian distribution, much as the symmetrized square root of the Marchenko--Pastur distribution is a Gaussian law in rectangular free probability.

 The finite free $R$-transform and the rectangular finite free $R$ transform were defined in \cite{Marcus2021} and \cite{Gribinski2024}, respectively. We will use the following convenient equivalent definitions for these transforms, see \cite[Lemma 6.4]{Marcus2021}\footnote{ We note that \cite{Marcus2021} is not consistent with where to truncate the series for the $R$-transform, but the mathematically consistent truncation should be $\mod s^{d}$.}, \cite[(3.1)]{Arizmendi-Perales2018} and \cite[Equation (14)]{Gribinski2024}. \begin{definition}\label{def:finite R transforms}
	Let $p$ be a degree $d$ polynomial. Let $P_{d}$ and $P_{d,n}$ be  formal power series  such that \begin{equation}\label{eq:finite Fourier transforms}
		P_{d}\left(\frac{D}{d}\right)x^d=p(x),\text{ and } P_{d,n}\left(\frac{M_n}{d(d+n)}\right)x^d=p(x).
	\end{equation} The finite free $R$-transform of $p$ is the truncated power series \begin{equation}\label{eq:finite free R}
		R_{p}^{d}(s):=-\frac{1}{d}\frac{P_{d}'(s)}{P_{d}(s)}\mod[s^{d}].
	\end{equation} Similarly, the $(d,n)$-rectangular finite free $R$-transform of $p$ is the truncated power series\begin{equation}\label{eq:rect finite R def}
		R_{p}^{d,n}(s):=-\frac{s}{d}\frac{P_{d,n}'(s)}{P_{d,n}(s)}\mod[s^{d+1}]
	\end{equation}
\end{definition}  Only the first $d+1$ coefficients of the formal power series $P_{d}$ and $P_{d,n}$ are determined by \eqref{eq:finite Fourier transforms}, and the higher degree terms can be defined arbitrarily. We could choose $P_{d}$ and $P_{d,n}$ to be degree $d$ polynomials, but  it will be helpful in some results to think of these as entire functions with real roots.  For formal power series $P_{d}$ and $P_{d,n}$ with degree $0$ coefficient $1$, which is the only case we consider in this paper, one should always interpret the logarithmic derivative on the right hand sides of \eqref{eq:finite free R} and \eqref{eq:rect finite R def} as formal power series \begin{equation}\label{eq:log der series def}
\frac{\dd}{\dd s}\log\left(P_{d}(s)\right)\text{ and }\frac{\dd}{\dd s}\log\left(P_{d,n}(s)\right)
\end{equation} where this derivative is taken coefficient-wise and we note for any sequences of complex number $(a_k)_{k\geq 1}$ and $(b_{k})_{k\geq 1}$ such that as formal power series \begin{equation}\label{eq:log series}
\log\left(1+\sum_{k=1}^{\infty}\frac{a_k}{k!}s^{k}\right)=\sum_{k=1}^\infty\frac{b_k}{k!}s^k,
\end{equation} we have the relationship \begin{equation}\label{eq:log series coeff}
b_k=\sum_{\pi\in\mathcal{P}(k)}a_{\pi}(-1)^{|\pi|-1}(|\pi|-1)!.
\end{equation} 	 Let $P$ be any formal power series with degree $0$ coefficient $1$ and let $L(s):=\log P(s)=\sum_{k=1}^{\infty}\frac{b_k}{k!}s^k$ be the logarithm of $P$, again considered as a formal power series. From \eqref{eq:log series coeff} we see that $b_k$ depends only on the degree less than or equal to $k$ series coefficients of $P$. Hence, taking the logarithmic derivative of a power series $P$, that has been truncated $\mod s^{d+1}$, and then truncating $\mod s^d$ is equivalent to taking the logarithmic derivative of the un-truncated power series $P$, and then truncating mod $s^{d}$. I.e.\ \begin{equation}\label{eq:exchanging log der and trunc}
\frac{\dd}{\dd s}\log\left(P(s)\mod s^{d+1} \right)\mod s^{d}=	\frac{\dd}{\dd s}\log\left(P(s) \right)\mod s^d.
\end{equation} Hence, $R_{p}^{d}$ ad $R_{p}^{d,n}$ depend only on the first $d+1$ coefficients of $P_{d}$ and $P_{d,n}$ respectively, and we can choose $P_{d}$ and $P_{d,n}$ to be infinite power series without affecting the $R$-transforms.

On the finite level, the $R$-transforms are defined as truncated power series, and the coefficients are exactly the finite free cumulants. So in some sense our approach is equivalent to considering cumulants. However, as will hopefully be clear in our main results, the $R$-transform approach provides quite a clean comparison between Appell sequences and ID distributions. 

\subsection{Brief background on Appell sequences and the Laguerre--P\'olya class} We will now present some background in the Laguerre--P\'olya class, Jensen polynomials, and Appell polynomials. We encourage the interested reader to consult \cite{Craven-Csordas1989,OSullivan2021,Dimitrov-Youssef2009,Dimitrov98} and the references therein for a more thorough treatment of these topics. An entire function $f$ is said to belong to the Laguerre--P\'olya class if it admits the representation \begin{equation}\label{eq:f product}
	f(z)=Cz^{m}e^{cz-\frac{\sigma^2}{2}z^2}\prod_{j=1}^{N}\left(1-\frac{z}{x_{j}}\right)e^{\frac{z}{x_j}},
\end{equation} where $C,c,\sigma\in\R$, $m\in\N$, $\{x_{j}\}_{j=1}^N\subset\R\setminus\{0\}$ is a multi-set, $\sum_{j=1}^{N}x_{j}^{-2}<\infty$, and $N\in\N\cup\{\infty\}$. We will also consider the series representation of an entire function $f$ \begin{equation}\label{eq:power series}
	f(z)=\sum_{k=0}^{\infty}\frac{\gamma_{k}}{k!}z^k.
\end{equation} We associate to $f$ a sequence of polynomials $\{J_{d,f}\}_{d=1}^\infty$, known as the Jensen polynomials, defined by \begin{equation}\label{eq:Jensen}
	J_{d,f}(z):=\sum_{k=0}^{d}\gamma_{k}\binom{d}{k}z^{k}. 
\end{equation} These polynomials have the remarkable property that $f$ belongs to the Laguerre--P\'olya class if and only if $J_{d,f}$ has only real zeros for every $d\in\N$. Moreover, the polynomials $J_{d,f}(z/d)$ converge to $f(z)$ uniformly on compact subsets.  Throughout the remainder of the paper we will assume $f(0)=1$, i.e.\ $C=1$ and $m=0$ in \eqref{eq:f product}.  We define the Appell polynomials $\{A_{d,f}\}_{d=1}^\infty$ associated to $f$ by \begin{equation}\label{eq:Appell}
	A_{d,f}(z):=\sum_{k=0}^{d}\gamma_{k}\binom{d}{k}z^{d-k}=z^{d}J_{d,f}\left(\frac{1}{z}\right).
\end{equation}  It is not hard to see that $A_{d,f}$ has only real roots if and only if $J_{d,f}$ does. The Appell polynomials of a function $f$ do in fact satisfy the relationship \eqref{eq:Appell def}. Conversely, we can also always associate an Appell sequence to  any  formal power series $f$ by \begin{equation}\label{eq:transform of Appell}
	A_{d,f}(z)=f\left(D\right)z^{d}.
\end{equation} This agrees with the definition of an Appell sequence of an entire function $f$. At this point it is worth noting that if $f(z)=e^{-\frac{z^2}{2}}$, then $A_{d,f}$ are the degree $d$ Hermite polynomials.

Appell sequences and the Laguerre--P\'olya class have previously appeared in the random matrix literature \cite{Assiotis-Najnudel2021,Assiotis2022} in a different, although possibly related context. In \cite{Assiotis-Najnudel2021} the authors consider, among other things, random \emph{interlacing} arrays of parameter $\theta\in (0,\infty]$. In particular, the $\theta=\infty$ case of \cite[Theorem 1.13]{Assiotis-Najnudel2021} classifies all real rooted Appell sequences. In \cite{Assiotis2022}, the author connects random Laguerre--P\'olya functions to unitarily invariant infinite dimensional random Hermitian matrices. We discuss how our work connects to these results in Section \ref{sec:further directions}.

\section{Main results}\label{sec:main results} 

\subsection{{$\boxplus$}-infinite divisibility and Appell sequences}\label{sec:ID and Appell} Before stating our main results, we justify the normalization required to get a non-degenerate limit. It is easy to check from \eqref{eq:finite free R} that \begin{equation}\label{eq:Appell R}
	R_{A_{d,f}}^{d}(s)=-\frac{f'(ds)}{f(ds)}\mod\left[s^{d}\right].
\end{equation} We thus want to normalize to get a change of variables $s\mapsto\frac{s}{d}$ in \eqref{eq:Appell R}. To achieve this we consider the polynomials \begin{equation}\label{eq:normalized Appell}
	\hat{A}_{d}(z)=[\mathcal{D}_{1/d}A_{d,f}]^{\boxplus_{d} d}(x)=f\left(\frac{D}{d}\right)^dz^d.
\end{equation} Note that for Hermite polynomials, where $f(z)=e^{-\frac{z^2}{2}}$, \eqref{eq:normalized Appell} is equivalent to dividing the roots by $\sqrt{d}$. \begin{theorem}\label{thm:Appell and free ID}
	Let $\hat{A}_d$ be defined as in \eqref{eq:normalized Appell} for $f$ of the form \eqref{eq:f product} with $f(0)=1$. Then $\mu_{\hat{A}_{d}}$ converges weakly as $d\rightarrow\infty$ to the $\boxplus$-ID distribution $\mu_f$ with $R$-transform \begin{equation}\label{eq:ID limit R}
		R_{\mu_f}(z)=\gamma+\sigma^2z+\int_{\R}\frac{z+t}{1-zt}\frac{t^2}{t^2+1}\dd\nu_f(t),
	\end{equation} where $\nu_f=\sum_{j=1}^\infty\delta_{1/x_j}$ and $\gamma=-c-\sum_{j=1}^\infty\frac{1}{x_j^3+x_j}$.
\end{theorem} As can be seen in \eqref{eq:Appell R}, even before any normalization $R_{A_{d,f}}^{d}$ is a truncated version of the $R$-transform of some $\boxplus$-ID distribution. The normalization \eqref{eq:normalized Appell} fixes the distribution in $d$, so that the truncated $R$-transform converges to this full $R$-transform. For $N<\infty$ in \eqref{eq:f product} one could prove Theorem \ref{thm:Appell and free ID} using \eqref{eq:additive fourier transform}, as  $\hat{A}_{d}$ is the finite free additive convolution of a Hermite polynomial (whose roots have been shifted and rescaled) and a finite number of Laguerre polynomials (whose roots have been shifted and rescaled). In fact, one may be able to extend this argument to the case when $N=\infty$ by establishing a finite free version of Kolmogorov's three-series theorem. We do not explore this direction here and instead employ the finite $R$-transform. Let us also remark that the parameter $\gamma$ in Theorem \ref{thm:Appell and free ID} should be interpreted as the mean of $\mu_{f}$, $-c$, plus a term which centers the contribution from the integral in \eqref{eq:ID limit R}. We include this extra term in $\gamma$ only to make the representation, using Theorem \ref{rthm:ID rep}, as an $\boxplus$-ID distribution clear. 

The following is essentially a corollary of the fact that the smallest roots of the Jensen polynomials of $f$ converge to the roots of $f$ in compact sets, but we present it here in the language of probability. \begin{theorem}\label{thm:Appell point proc}
Let $A_{d,f}$ be as in \eqref{eq:transform of Appell} and let $\tilde{A}_{d}:=\mathcal{D}_{1/d}A_{d,f}$. Then, $d{t^2}\dd\mu_{\tilde{A}_{d}}(t)$ converges weakly as a Radon measure to \begin{equation}\label{eq:Gf measure}
	G_{f}=\sigma^{2}\delta_{0}+\sum_{j=1}^\infty\frac{1}{x_{j}^2} \delta_{1/x_j}.
\end{equation}
\end{theorem}

This convergence of measures and the relationship to infinitely divisible distributions motivates the following description of the domains of attraction of Appell sequences under repeated differentiation.  \begin{theorem}\label{thm:Appell domain}
	Let $\{p_{d}\}_{d=1}^\infty$ be a sequence of monic real rooted polynomials indexed by their degree such that $m_{1}(p_{d})\rightarrow c$, $m_{2}(p_d)=\Theta(d)$, and \begin{equation}\label{eq:domain theorem pp conv}
		dt^2\dd\mu_{\mathcal{D}_{d^{-1}}p_{d} }(t)\rightarrow \dd G_{f}(t)
	\end{equation} weakly as Radon measures as $d\rightarrow\infty$ where $G_{f}$ is the measure in \eqref{eq:Gf measure} associated to  the function  \begin{equation}
	f(z)=e^{-cz-\frac{\sigma^2}{2}z^2}\prod_{j=1}^{\infty}\left(1-\frac{z}{x_{j}}\right)e^{\frac{z}{x_{j}}}.
\end{equation}  Then, for any $\ell\in\N$\begin{equation}\label{eq:conv to Appell}
	\lim\limits_{d\rightarrow\infty}\frac{\ell!}{d!}D^{d-\ell}p_{d}(x)=f\left(D\right)x^{\ell},
\end{equation} uniformly on compact subsets as $d\rightarrow\infty$. 
\end{theorem}
It follows from \cite[Theorem 3.4]{Bercovici-Pata1999} that the conditions of Theorem \ref{thm:Appell domain} are exactly the necessary and sufficient conditions for $\mu_{\mathcal{D}_{d^{-1}}p_d }$ to be in the domain of attraction of $\mu_{f}$ in Theorem \ref{thm:Appell and free ID}, and Theorem \ref{thm:Appell domain} is presented in this probabilistic manor. Moreover, the conditions of Theorem \ref{thm:Appell domain} also have a natural interpretation in terms of polynomials and the Laguerre--P\'olya class. The roots of the sequence $\{p_d\}$ satisfy what is referred to in \cite{Assiotis-Najnudel2021,Assiotis2022} as the Olshanski--Vershik conditions after \cite{Olshanki-Vershik1996}. Thus, from \cite[Propisition 2.5]{Assiotis2022}, a real rooted sequence of polynomials $\{p_d\}_{d=1}^{\infty}$ satisfies the assumptions of Theorem \ref{thm:Appell domain} if and only if the reversed sequence $q_{d}(z)=z^{d}p_{d}\left(1/z\right)$ converges uniformly on compact subsets to $f$. We discuss further connections between our results and previous work on random entire functions and random matrices in Section \ref{sec:further directions}.

The normalization in \eqref{eq:normalized Appell} is one way to recover a non-degenerate $\boxplus$-ID distribution from an Appell sequence, and is such that the limiting measure has cumulants of all orders and the $R$-transform has a series expansion at the origin. However, as we explain in Section \ref{sec:heavy-tailed limit} it appears that the ERMs of Appell sequences, under a different rescaling, should converge to $\boxplus$-stable distributions. The barrier to making the connection rigorous is that the true $R$-transform of $\mu_{A_{d}}$ is hard to access, and it is hard to describe $\mu_{A_{d}}$ in terms of our finite versions without implicitly or explicitly using cumulants. We do however have the following result which is illustrative of how one may overcome this issue in future work. \begin{theorem}\label{thm:Cauchy distribution and cosine}
	Let $f(z)=\cos(z)$, and $A_{d}(z)=A_{d,f}(z)$. Then, $\mu_{A_d}$ converges weakly as $d\rightarrow\infty$ to the standard Cauchy distribution, which is $\boxplus$-stable.
\end{theorem} 

\begin{remark}\label{rmk:free decomp}
		Below we collect some important examples and remark on how to interpret the limiting distribution for general $f$. \begin{enumerate}
		\item When $f(z)=e^{-z^2/2}$ Theorem \ref{thm:Appell and free ID} recovers the convergence of the ERM of Hermite polynomial roots to the semicircle law. For this choice of $f$, Theorem \ref{thm:Appell domain} recovers the Hermite limits of \cite{Gorin-Klepttsyn2020universal,Hoskins-Steinerberger2022,Campbell-ORourke-Renfrew2024even}. $f(z)=e^{-cz}$ corresponds to the degenerate distribution $\mu_f=\delta_{c}$, and Theorem \ref{thm:Appell domain} recovers the law of large numbers in \cite{Campbell-ORourke-Renfrew2024even}. 
		
		
		\item If $f(z)=(1-z)^{\lambda}$ for some $\lambda\in\N$, then $\mu_{f}$ is a Marchenko--Pastur distribution with rate $\lambda$. More generally, if $f(z)=(1-az)^{\lambda}$ then $\mu_{f}$ is a Marchenko--Pastur with rate $\lambda$ multiplied by $a$. 
		
		\item For general $f$, $\mu_{f}$ is the free additive convolution of a constant, a semicircle law, and an infinite number of shifted and scaled Marchenko--Pastur distributions. 
		
		\item In the language of free L\'evy processes, the decomposition of $f$ determines the L\'evy-Khintchine triple of the process, where the exponential term in \eqref{eq:f product} determines the free Brownian motion with drift, and the roots of $f$ determine the jump process in the L\'evy--It\^{o} decomposition. 
		
%
	\end{enumerate}
\end{remark}

\subsection{Fractional rectangular convolution powers}\label{sec:fract rect conv}

Our first main result on rectangular free convolutions, Theorem \ref{thm:Limit theorem}, captures the behavior of roots under repeated application of the operator $M_n$ on all possible scales. To simplify presentation we work under the following assumption. 

	\begin{assumption}\label{assumption:limiting measure}
		Let $\{p_{d}\}_{d=1}^\infty$ be a sequence of polynomials with only real roots indexed by degree. We say that $\{p_d\}$ satisfies Assumption \ref{assumption:limiting measure} if there exists a  compactly supported probability measure $\mu_{p}$ such that $\mu_{p_d}$ converges weakly to $\mu_{p}$ as $d\rightarrow\infty$ and if $\{x_{j,d}\}_{1\leq j\leq d}$ denotes the triangular array of roots of $\{p_d\}$, then \begin{equation}\label{eq:uniform root bound}
			\sup_{1\leq j\leq d<\infty}|x_{j,d}|\leq C,
		\end{equation} for some $C>0$.
	\end{assumption}
	
	We break Theorem \ref{thm:Limit theorem} up into three cases, based on the number of times $M_n$ is applied.

	\begin{theorem}\label{thm:Limit theorem}
		Let $\{p_{d}\}_{d=1}^\infty$ be a sequence of polynomials with only non-negative roots satisfying Assumption \ref{assumption:limiting measure} with some limiting root measure $\mu_{p}$. For any $j\in\{1,\dots,d\}$, let\begin{equation}\label{eq: rect pj def}
			p_{j,d}(x)=\frac{1}{(d)_{j}(n+d)_{j}}M_{n}^{j}p_{d}(x).
		\end{equation}
		
		Assume $\ell=d-j$ is fixed:\begin{enumerate}
			\item If $n>-1$ is fixed and the first moment of $\mu_{p}$ is $1$, then \begin{equation}\label{eq:Lag limit}
				\lim\limits_{d\rightarrow\infty}\mathcal{D}_{{d+n}}p_{j,d}(x)\approx L_{\ell}^{(n)}\left(x\right),
			\end{equation} uniformly on compact subsets.
			\item If $n\rightarrow\infty$ as $d\rightarrow\infty$, and the first moment of $\mu_{p}$ is $1$, then \begin{equation}
				\lim\limits_{d\rightarrow\infty}\mathcal{D}_{\frac{d+n}{n}}p_{j,d}(x)=(z-1)^{\ell},
			\end{equation} uniformly on compact subsets.
		\end{enumerate} 
	
		Next assume $\ell=d-j\rightarrow\infty$ and $\ell=o(d)$:\begin{enumerate}
			\item If $-1<n=o(\ell)$ and the first moment of $\mu_{p}$ is $1$, then weakly in the sense of probability measures \begin{equation}
				\lim\limits_{d\rightarrow\infty}\mu_{\mathcal{D}_{\frac{d+n}{n+\ell}}p_{j,d}}=\nu_{1},
			\end{equation} where $\nu_{1}$ is the Marchenko-Pastur distribution with rate $1$ and mean $1$. 
			\item If $n\sim \alpha\ell$ for some $\alpha\in(0,\infty)$, and the first moment of $\mu_{p}$ is $1$, then \begin{equation}
				\lim\limits_{d\rightarrow\infty}\mu_{\mathcal{D}_{\frac{d+n}{n+\ell}}p_{j,d}}=\nu_{\alpha},
			\end{equation} where $\nu_{\alpha}$ is the push-forward of the Marchenko-Pastur distribution with rate ${1+\alpha}$ by $t\mapsto \frac{1}{1+\alpha}t$. 
		\end{enumerate}
	
		Finally, assume $d-j\sim t d$ for $t\in (0,1)$:\begin{enumerate}
			\item If $-1<n=o(d)$, then weakly in the sense of probability measures \begin{equation}
				\lim\limits_{d\rightarrow\infty}\sqrt{\mu_{p_{j,d}}}=\mathcal{D}_{t}\sqrt{\mu_{p}}^{\boxplus_{1}\frac{1}{t}}.
			\end{equation}  
			\item If $n\sim\alpha d$, then weakly in the sense of probability measures $\sqrt{\mu_{p_{j,d}}}$ converges to a probability measure $\mu_{\alpha,t}$ with $\boxplus_{\frac{t}{t+\alpha}}$ rectangular $R$ transform $C_{t,\alpha}$ satisfying \begin{equation}\label{eq:mixed ratio limit}
				C_{t,\alpha}(z)=\frac{1}{t}C_{\alpha}\left(t\frac{\alpha+t}{\alpha+1}z\right),
			\end{equation}  where $C_{\alpha}$ is the $\boxplus_{\frac{1}{1+\alpha}}$ rectangular $R$-transform of $\sqrt{\mu_p}$. 
		\end{enumerate}
	\end{theorem}

	\begin{remark}\label{rem:relaxing assumptions}
		The conditions of some parts of Theorem \ref{thm:Limit theorem} may be relaxed. For example, \eqref{eq:Lag limit} already appeared in \cite{Campbell-ORourke-Renfrew2024even} under optimal moment conditions on the polynomials and no assumption of the root measures converging. We expect that the compact support, and even some moment assumptions, on the limiting measure $\mu_{p}$ may be removed and analogous versions of this theorem should hold, either with no change in the limit, or with the Marchenko--Pastur distribution replaced with some other $\boxplus_{\lambda}$-ID distribution. However, while the finite $R$-transforms present a more analytic flavor, they are limited to measures with all moments finite by nature of being defined as formal power series. 
	\end{remark}

	To prove Theorem \ref{thm:Limit theorem} we will track the convergence of the \emph{rectangular free $R$-transform} as $d\rightarrow\infty$. The following lemma characterizes how the $R$-transform evolves under $D$ and $M_n$. \begin{lemma}\label{lem:R transform lemma}
		Let $p$ be real rooted and ${p}_{j}(x)=\frac{(d-j)!}{d!}p^{(j)}(x)$. Then, \begin{equation}\label{eq:finite R red}
			R_{{p}_j}^{d-j}(s)=R_{p}^{d}\left(\frac{d-j}{d}s\right)\mod[s^{d-j}].
		\end{equation} Moreover, if $p$ has non-negative roots and $\hat{p}_j(x)=\frac{1}{(d)_{j}(n+d)_j}M_{n}^{j}p(x)$, then\begin{equation}\label{eq:rect finite R red}
		R_{\hat{p}_j}^{d-j,n}(s)=\frac{d}{d-j}R_{p}^{d,n}\left(\frac{(d-j)(n+d-j)}{d(n+d)}s\right)\mod[s^{d-j+1}].
	\end{equation} Moreover, for any polynomial $p$ \begin{equation}\label{eq:rect to square}
	\lim\limits_{n\rightarrow\infty} R^{d,n}_{p}(s)=sR_{p}^{d}(s)\mod\left[s^{d+1} \right].
\end{equation}
	\end{lemma} 
	After noting that the coefficients of the finite free $R$-transform are in fact the finite free cumulants, \eqref{eq:finite R red} follows from \cite[Proposition 3.4]{Arizmendi-Fujie-Perales-Ueda2024}. We provide a short independent proof using Definition \ref{def:finite R transforms}. We give the rectangular analogue of \cite[Proposition 3.4]{Arizmendi-Fujie-Perales-Ueda2024} below. The finite free $(n,d)$-cumulants as defined in \cite{Cuenca24} are\footnote{In \cite{Cuenca24} the definition differs slightly, however the equivalence to \eqref{eq:unscaled finite rect cumulants} can be seen  from \eqref{eq:log series coeff}.} \begin{equation}\label{eq:unscaled finite rect cumulants}
	K_{2k}^{n,d}[p]=\frac{(-1)^{k}}{(k-1)!}\sum_{\pi\in\mathcal{P}(k)}(-1)^{|\pi|-1}(|\pi|-1)!\frac{N!_{\pi}a_{\pi}}{(d)_{\pi}(n+d)_{\pi}}.
\end{equation} We will also consider the rescaled version \begin{equation}\label{eq:scaled finite rect cumulants}
\kappa_{2k}^{n,d}[p]:=-(d)^{2k-1}\left(1+\frac{n}{d} \right)^{k}K_{2k}^{n,d}[p].
\end{equation}

\begin{lemma}\label{thm:Cumulant theorem}
	For any $n>-1$, $j\in\{0,1,\dots,d\}$ and any $k\in\{1,\dots, d-j\}$\begin{equation}\label{eq:cumulant formula Mn}
		K_{2k}^{n,d-j}\left[\frac{1}{(d)_{j}(n+d)_{j}}M_{n}^{j}p\right]=K_{2k}^{n,d}[p],
	\end{equation}and
	
	\begin{equation}\label{eq:scaled cumulant formula Mn}
		\kappa_{2k}^{n,d-j}\left[\frac{1}{(d)_{j}(n+d)_{j}}M_{n}^{j}p\right]=\left(\frac{d-j}{d} \right)^{2k-1}\kappa_{2k}^{n,d}[p].
	\end{equation}
\end{lemma}

\subsection{$\boxplus_{\lambda}$-infinitely divisible distributions and Appell like polynomials}\label{sec:rect ID and L_Appell} We now work towards an analogue of Theorem \ref{thm:Appell and free ID} for the rectangular free convolution. To do so we define variations of the Jensen and Appell polynomials for the operator $M_n$. Let $g$ be a function such that \begin{equation}\label{eq:g series and prod}
	g(z)=e^{-\sigma^2z}\prod_{j=1}^{\infty}\left(1-\frac{z}{\alpha_j^2}\right)=\sum_{k=0}^{\infty}\frac{\eta_{k}}{k!}z^{k},
\end{equation} for $\sigma\geq 0$, $\alpha_j\in\R$, and \begin{equation}
	\sum_{j=1}^\infty\frac{1}{\alpha_j^2}<\infty.
\end{equation} Such a function is said to be in $\mathcal{LP}I$. We define the index $n$ Laguerre--Appell polynomials associated to $g$ as the sequence of polynomials $\{L_{d,g}\}_{d=1}^\infty$ such that\begin{equation}\label{eq:LA polnyomials}
	L_{d,g}(z):=\sum_{k=0}^d\eta_{k}(n+d)_{k}\binom{d}{k}z^{d-k} =g\left(M_n\right)z^{d}.
\end{equation} We also define the index $n$ Laguerre--Jensen polynomials $\{K_{d,g} \}$ as\begin{equation}
	K_{d,g}(z):=\sum_{k=0}^d\eta_{k}(n+d)_{k}\binom{d}{k}z^{k}=z^{d}L_{d,g}\left(\frac{1}{z} \right).
\end{equation} We then have the following lemma. \begin{lemma}\label{lem:LP1 test}
	Let $n>-1$ and $g$ be a formal power series with index $n$ Laguerre--Appell and Laguerre--Jensen polynomials $\{L_{d,g}\}$ and $\{K_{d,g}\}$. Then, $g$ is in $\mathcal{LP}I$ if and only if $K_{d,g}$ has only positive roots for any $d\in\N$. Moreover, \begin{equation}
		\lim\limits_{d\rightarrow\infty}K_{d,g}\left(\frac{z}{d(n+d)}\right)=g(z),
	\end{equation} uniformly on compact subsets.
\end{lemma}

Again, we define the normalized version of our Laguerre--Appell polynomials:\begin{equation}\label{eq:normalized L-Appell}
	\hat{L}_{d}(z)=g\left(\frac{M_n}{d(n+d)}\right)^{d}z^d.
\end{equation}  We are ready to state the theorems. \begin{theorem}\label{thm:L-A and Rect ID}
	Assume $\frac{n+d}{d}\rightarrow\lambda\in(0,1]$ and let $\hat{L}_{d}$ be defined as in \eqref{eq:normalized L-Appell}. Then, $\sqrt{\mu_{\hat{L}_{d}}}$ converges weakly as $d\rightarrow\infty$ to a $\boxplus_{\lambda}$-ID distribution $\tilde\mu_{g}$ with rectangular $R$-transform \begin{equation}
		C_{\tilde{\mu}_{g}}^{\lambda}(z)=z\int_{\R}\frac{t^2+1}{1-zt^2}\dd G_g(t), 
	\end{equation}where\begin{equation}
		G_g=\sigma^2\delta_{0}+\sum_{j=1}^{\infty}\frac{1}{\alpha_{j}^2+1}\frac{1}{2}\left[\delta_{1/\alpha_{j}}+\delta_{-1/\alpha_{j}} \right].
	\end{equation}  
\end{theorem}

	\begin{theorem}\label{thm:L-Appell point process conv}
		Let $L_{d,g}$ be as in \eqref{eq:LA polnyomials} and let $\tilde{L}_{d,g}=\mathcal{D}_{d^{-1}(n+d)^{-1}}L_{d,g}$. Then, $d{t}\dd\mu_{\tilde{L}_{d,g}}(t)$ converges weakly as a Radon measure to \begin{equation}\label{eq:Gg def}
			G_{g}=\sigma^{2}\delta_{0}+\sum_{j=1}^{\infty}\frac{1}{\alpha_{j}^2}\delta_{1/\alpha_{j}^2}.
		\end{equation} 
	\end{theorem}

	\begin{theorem}\label{thm: rect finite ID domain of attraction}
		Let $n>-1$ be fixed and let $\{p_{d}\}_{d\geq 1}$ be a sequence of monic polynomials with non-negative roots indexed by their degree such that $m_{1}(p_d)/d\rightarrow\tau^{2}$ for some $\tau>0$ and \begin{equation}
			dt\dd\mu_{\mathcal{D}_{d}p_d}(t)\rightarrow \dd G_{g}(t),
		\end{equation} as Radon measures where $G_g=\sigma^{2}\delta_{0}+\sum_{j=1}^\infty\frac{1}{\alpha_{j}^2}\delta_{\alpha_{j}^{2}}$. Then, for any $\ell\in\N$ \begin{equation}
		\lim\limits_{d\rightarrow\infty}\frac{1}{(d)_{d-\ell}(n+d)_{d-\ell}}M_{n}^{d-\ell}p_d(x)=g(M_n)x^{\ell},
	\end{equation} uniformly on compact subsets as $d\rightarrow\infty$, where \begin{equation}
	g(z)=e^{-\sigma^2z}\prod_{j=1}^{\infty}\left(1-\frac{z}{\alpha_j^2}\right).
\end{equation} 
	\end{theorem}


%

	\section{Novelties, connections, and further directions}\label{sec:further directions} We summarize the novelties of our results, connections to previous works, some open questions, and possible future directions of the ideas outlined here.
	
	\subsection{Novelties} While the (rectangular) finite $R$-transform is used often in \cite{Marcus2021,Gribinski2024}, the recent finite free probability results \cite{Arizmendi-Fujie-Perales-Ueda2024,Arizmendi-GarzaVargas-Perales2023,Campbell-ORourke-Renfrew2024even} centered around Question \ref{ques:Root question}, as well as other results in finite free probability, tend to rely on computations of finite free cumulants. In particular, \cite[Proposition 3.4]{Arizmendi-Fujie-Perales-Ueda2024} demonstrates that the cumulants evolve nicely under repeated differentiation. Our approach, in particular the proof of Lemma \ref{lem:R transform lemma}, demonstrates that the same nice evolution can be seen for $R$-transforms, for both the operator $D$ and $M_n$, without ever referring to cumulants. Using the $R$-transform approach also illuminates the connections to infinite divisibility, where the L\'evy--Khintchine representation theorem and Bercovici–Pata bijection give an explicit characterization of the $R$-transform while the free cumulants (to the best of our knowledge) fail to have such a nice representation theorem. In Section \ref{sec:proofs} the $R$-transform approach allows us to prove results for $D$ and $M_n$ in parallel, with little to no changes from one theorem to the next.

	\subsection{Relationship to the $\beta=\infty$-corners process} In \cite{Gorin-Marcus2020,Cuenca2021,Assiotis-Najnudel2021,Gorin-Klepttsyn2020universal} the authors consider the \emph{$\beta$-corners} or \emph{orbital $\beta$-process}. This is an array of real random interlacing layers $\{a_{k}^{(i)}\}_{1\leq k\leq i\leq d}$, i.e.\ for any $i\geq 1$, \begin{equation*}
		a_{1}^{(i+1)}\leq a_{1}^{(i)}\leq a_{2}^{(i+1)}\leq\cdots\leq a_{i}^{(i+1)}\leq a_{i}^{(i)}\leq a_{i+1}^{(i+1)},
	\end{equation*} whose distribution is explicit, given the top layer $a_{1}^{(d)},\dots,a_{d}^{(d)}$ and depends on $\beta\in (0,\infty]$. For $\beta=1$ or $2$, this is the joint distributions of eigenvalues for the top-left corners of orthogonally or unitarily invariant random matrices with eigenvalues $a_{1}^{(d)},\dots,a_{d}^{(d)}$. In \cite{Gorin-Marcus2020} the authors prove that this processes converges, as $\beta\rightarrow\infty$ to repeated differentiation. In \cite{Cuenca2021} the author provides a universal limit of the $\ell$-th layer as $d\rightarrow\infty$ for $\beta<\infty$ under some growth assumptions on the top layer. The $\beta=\infty$ version of this result is proved in \cite[Theorem 2.9]{Gorin-Klepttsyn2020universal}   using counter integration, and by \cite[Theorem 2.8]{Campbell-ORourke-Renfrew2024even}\footnote{The authors of \cite{Campbell-ORourke-Renfrew2024even} were not aware of the result in \cite{Gorin-Klepttsyn2020universal} until recently and this will be pointed out in a future update to \cite{Campbell-ORourke-Renfrew2024even}.} using different purely finite free probabilistic techniques.  Theorem \ref{thm:Appell domain} can be thought of as an extension of these $\beta=\infty$ results to top layers outside the domain of attraction of Hermite/Gaussian limits.

	 The work of \cite{Assiotis-Najnudel2021} defines what it means for an infinite array of interlacing layers $\{a_{k}^{(i)}\}_{1\leq k\leq i\leq \infty}$ to be \emph{consistent} with parameter $\beta$. An infinite family is consistent when any finite cut off, $\{a_{k}^{(i)}\}_{1\leq k\leq i\leq d}$, is a $\beta$-corner process. They define an explicit bijection between the consistent distributions on infinite interlacing arrays and probability measures on the space of Laguerre--P\'olya functions. The $\beta=\infty$ version of their result \cite[Theorem 1.13]{Assiotis-Najnudel2021} is exactly a characterization of the real rooted Appell sequences. The paper \cite{Assiotis2022} considers these ideas in the context of infinite dimensional unitarily invariant  random Hermitian matrices, and again provides an explicit bijection between these random matrices and random Laguerre--P\'olya functions. When we consider that projection onto corners of random matrices and free convolution powers are essentially equivalent in the large dimensional limit, see for example \cite{Shlaykhtenko-Tao2020,Nica-Speicher1996,Campbell-ORourke-Renfrew2024}, we can consider our results on free infinite divisibility and Appell sequences as a free probabilistic interpretation of consistent families. The decomposition mentioned in Remark \ref{rmk:free decomp} for the measure in Theorem \ref{thm:Appell and free ID} is then analogous to the decomposition of infinite dimensional unitarily invariant random Hermitian matrices, see for example \cite[(8)]{Assiotis2022} and \cite{Olshanki-Vershik1996} for the original result. It is worth noting, that in a somewhat similar vein, \cite{Bufetov-Gorin2015} related characters of the infinite dimensional unitary group to probability measures which are infinitely divisible with respect to a \emph{quantized free convolution}. 
	 
	 The operator $M_n$ also has an interpretation as a $\beta=\infty$-corner process. If we assume $A$ is a $d\times (d+n)$ rectangular random matrix with prescribed singular values, which is distributionally invariant under unitary multiplication (of an appropriate dimension) on both the left and the right, then we can view the matrix $UAO$ in Lemma \ref{lem:rect matrix compression} as the top left $(d-1)\times (d-1+n)$ corner of $A$. Thus, we can view repeated application of $M_n$ as the $\beta=\infty$ version of this corner of rectangular random matrices. Similar corner processes appear in \cite{Borodin-Gorin2015,lernerbrecher2023hard,sun2016matrix}, however they do not appear to be exactly the same. The points in these processes exactly interlace, and in \cite{lernerbrecher2023hard} the process is obtained by removing columns of the rectangular matrix, but not also the rows. A $\beta$ deformed version of rectangular finite free addition was also considered by \cite{xu2023rectangular} where deformed rectangular free cumulants were defined.
	 
	 In general, it would be interesting to explore through $\beta$ deformed convolutions as in \cite{Gorin-Marcus2020,xu2023rectangular} the interplay between finite free probability, $\beta$ ensembles, infinite divisibility, and infinite dimensional random matrix theory.
	
	\subsection{A true analytic theory of finite free probability}
	 We next present the conjecture heuristically worked out in Section \ref{sec:heavy-tailed limit}. \begin{conjecture}
		Let $f$ be in the Laguerre--P\'olya class and satisfy \eqref{eq:RV roots}. Let $A_{d}(x)=f(D)x^{d}$, and let $\rho=\frac{1}{\alpha}-1$. Then, $\mu_{\mathcal{D}_{h(d)^{-1}d^{\rho}} A_{d} }$ converges weakly as $d\rightarrow\infty$ to a $\boxplus$-stable distribution with stability parameter $\alpha$. 
	\end{conjecture} We discuss the challenges of this problem in detail and one potential solution in Section \ref{sec:heavy-tailed limit}. More generally it motivates the following larger project.

	 If one could define an analytic version of the finite free $R$-transform, then it would be possible to capture heavy-tailed measures in the large degree limit. Another possible use of such a theory would be an improved version of the finite free Berry--Esseen theorem of \cite{Arizmendi-Perales2020}. The degree dependence of their bound grows quite fast, however this is most likely an artifact of the proof using finite free cumulants. It is common, in both the free and classical setting, for the strongest version of Berry--Esseen to be proven using analytic tools. An improved Berry--Esseen estimate would allow for the study of local properties of roots using finite free probability.
	
	We propose one new version of the $R$-transform for Appell sequences using the Jensen polynomials. Another more general approach would be to simply consider the finite free $R$-transform $R_{p}(z)=-\frac{1}{d}\frac{P_{d}'(z)}{P_{d}(z)}$ as a meromorphic function where $P_{d}$ is a degree $d$ polynomial such that $P_{d}(D/d)x^{d}=p(x)$. For Appell polynomials the difference between this approach and the approach we take in Section \ref{sec:heavy-tailed limit} is essentially the difference between approximating $f$ with Taylor polynomials or the real rooted Jensen polynomials. 	

	\subsection{The multiplicative free convolution} There also exists a multiplicative free convolution $\boxtimes$ and a finite free multiplicative convolution $\boxtimes_{d}$. The theory of $\boxtimes$-infinite divisibility has been well studied, see for example \cite{Bercovici-Voiculescu1993,Arizmendi-Hasebe2018}. The question remains open as to whether the ideas here can be extended to this setting, either on $[0,\infty)$ or by considering polynomials with roots on the unit circle. 
	
	\subsection{The geometry of polynomial roots} As mentioned in the introduction a fundamental problem in the geometry of polynomial roots is determining which operators on polynomials preserve the property of all roots belonging to a certain set. This was answered for roots in $\R$ in \cite{Borcea-Branden2009}. Our approach to $\boxplus_{d}$ and $\rfc{d}{n}$ is to associate them to an operator, $D$ or $M_n$, and to consider these convolutions as generated by these operators. This results in essentially identical definitions of the finite free $R$-transforms and proofs of our main results. It might be interesting to consider a general perspective. Given an operator $\mathcal{L}$ on polynomials, construct a convolution on polynomials generated from this operator and consider its interpretation as a probability theory. Are there operators which give finite versions of other non-commutative probability theories, such as the Boolean or monotone convolutions? The approach in \cite{Cuenca24} essentially lays out a general approach to defining cumulants. Additionally, can properties of this convolution be connected to the problem considered in \cite{Borcea-Branden2009}?
	
	\subsection{L\'evy flows on polynomials} In \cite{Hall-Ho-Jalowy-Kabluchko2023,Kabluchko2022leeyang} the authors consider the \emph{backwards heat flow} for polynomials, i.e.\ for a polynomial $p(z)$ they consider the dynamics for roots of $e^{-tD^2}p(z)$. These polynomial root flows have additionally appeared in \cite{Assoitis23} in the solutions of interacting particle systems. For random polynomials with independent coefficients the authors of \cite{Hall-Ho-Jalowy-Kabluchko2023} identified a critical time for the limiting empirical root measure of the polynomial to move from supported on $\C$ to $\R$ and to be a rescaling of the semicircle distribution. Let $f$ be a function in the Laguerre--P\'olya class of the form \eqref{eq:f product}. For $t\in [0,\infty)$, does a similar transition occur for $f\left(D/d\right)^{\lfloor td\rfloor}p_{d}(z)$? If $p_{d}(z)=z^{d}$, then for any fixed $t>0$, Theorem \ref{thm:Appell and free ID} could be adapted to show that the limiting root measure is the law at time $t$ of a free L\'evy process. 

	\subsection{Finite free probability and the roots of entire functions} As we have demonstrated in our main results the Appell and Jensen polynomials of functions in the Laguerre--P\'olya class serve as important examples in finite free probability, and through their infinite degree limit can be connected to nice distributions in free probability. These Jensen and Appell polynomials have also attracted recent attention \cite{Griffin-Ono-Rolen-Zagier2019,OSullivan2021,OSullivan2022Monat,Griffin-South2023,Farmer2022,Campbell-ORourke-Renfrew2024even} for their connection to the Riemann Hypothesis which dates back to \cite{Polya1927,Jensen1913}. Much of the renewed interest in these polynomials is due to the result of Griffin, Ono, Rolen, and Zagier \cite{Griffin-Ono-Rolen-Zagier2019} who established conditions on the coefficients of $f$ for the Jensen polynomials to converge to Hermite polynomials under repeated differentiation, without necessarily assuming the function $f$ has only real roots. The analogous result formulated in terms of roots of $f$ is given in \cite{Campbell-ORourke-Renfrew2024even} with the added assumption of $f$ having only real roots. The concept of the Laguerre--P\'olya class has also recently been generalized \cite{Wagner2022} to allow for some complex roots. By considering this more general class it may be possible to extend Theorem \ref{thm:Appell domain} to polynomials with complex roots. This may be one approach to unifying the results of \cite{Griffin-Ono-Rolen-Zagier2019} and \cite{Campbell-ORourke-Renfrew2024even} through some sort of \emph{weakly non-Hermitian finite free probability}, i.e.\ while the Jensen polynomials may have some complex roots, the rescaled point process will still converge to a point mass at $0$ resulting in a purely Hermite limit. More generally, it is worth continuing to investigate the role finite free probability may play in the study of entire functions.

	\section{Proofs}\label{sec:proofs} In this section we collect the proofs of our main results. We begin with the following lemma, which explicitly combines the results of \cite{Arizmendi-Perales2018,Cuenca24,Gribinski2024,Marcus2021} on convergence of ERMs and convergence of finite cumulants/$R$-transforms. 
	\begin{lemma}\label{lem:R-transform cumulant equivalence}
		Let $\mu$ be a probability measure with compact support and let $\{p_d\}$ be a sequence of real rooted polynomials  indexed by their degree satisfying \eqref{eq:uniform root bound}. Then, the following are equivalent:\begin{enumerate}
			\item The finite free $R$-transforms $R_{p_d}^{d}$ converge coefficient-wise as $d\rightarrow\infty$ to the power series expansion of $R_{\mu}$ around zero.
			\item $\mu_{p_d}$ converges weakly to $\mu$ as $d\rightarrow\infty$. 
		\end{enumerate} Additionally, if all the roots of $p_{d}$ are non-negative and $\mu$ is supported on $[0,\infty)$, then the following are equivalent: \begin{enumerate}
		\item The $(d,n)$-rectangular finite free $R$-transforms $R_{p_d}^{d,n}$ converge coefficient-wise as $d\rightarrow\infty$, $\frac{d}{d+n}\rightarrow\lambda\in (0,1]$ to the power series expansion of $C_{\sqrt{\mu}}^{\lambda}$ around zero.
		\item $\mu_{p_d}$ converges weakly to $\mu$ as $d\rightarrow\infty$.
	\end{enumerate}
	\end{lemma} 
	\begin{proof}
		The coefficients of the finite free $R$-transform and $(d,n)$-rectangular $R$ transform are exactly the finite free cumulants and $(d,n)$-rectangular finite free cumulants respectively. The first assertion is implicit in the work of \cite{Arizmendi-Perales2018} and is pointed out explicitly in \cite[Proposition 2.17]{Arizmendi-Fujie-Perales-Ueda2024}.  As the authors of \cite{Arizmendi-Fujie-Perales-Ueda2024} do not give a full proof we explain how this is implicit from \cite{Arizmendi-Perales2018} here. 
		
		Assume (1), i.e.\ for any $j$, $\kappa_{j}^{(d)}(p_d)\rightarrow\kappa_{j}(\mu)$ where $\kappa_{j}^{(d)}$ is the degree $j-1$ coefficient of $R_{p_d}^{d}$ and $\kappa_{j}(\mu)$ is the $j$-th free cumulant of $\mu$, see for example \cite[Chapter 2, (2.16)]{Mingo-Speicher2017} for an implicit definition of the free cumulants. Let $j\in\N$ and let $m_{j}(\mu_{p_d})$ denote the $j$-th moment of $\mu_{p_d}$. Additionally denote by $m_{j}(\mu)$ the $j$-th moment of $\mu$. From \cite[Theorem 5.4 and its proof]{Arizmendi-Perales2018} and \cite[(2.16)]{Mingo-Speicher2017} there exists monic polynomials $Q_{\sigma}$, indexed by partitions and of degree $d^{j+1-|\sigma|}$, such that \begin{equation}\label{eq:moment conv formula}
			\begin{aligned}
				\lim\limits_{d\rightarrow\infty}m_{j}(\mu_{p_d})-m_j(\mu)&=\lim\limits_{d\rightarrow\infty}\sum_{\sigma\in\mathcal{NC}(j)}\frac{Q_{\sigma(d)}}{d^{j+1-|\sigma|}}\kappa_{\sigma}^{(d)}(p_{d})-\kappa_{\sigma}(\mu)\\
				&=\lim\limits_{d\rightarrow\infty}\sum_{\sigma\in\mathcal{NC}(j)}\left(\frac{Q_{\sigma(d)}}{d^{j+1-|\sigma|}}-1\right)\kappa_{\sigma}^{(d)}(p_{d})+\kappa_{\sigma}^{(d)}(p_d)-\kappa_{\sigma}(\mu)\\
				&=0,
			\end{aligned}
		\end{equation} where $\mathcal{NC}(j)$ denotes the non-crossing partitions of $j$, see for example \cite[Definition 2.3]{Arizmendi-Perales2018}. As $j$ was arbitrary and the measure $\mu$ is determined by its moments, this completes the proof of $(1)\Rightarrow(2)$. The direction $(2)\Rightarrow(1)$ is similar. As the support of the measures are uniformly bounded weak convergence implies convergence of moments. Note that $\kappa_{1_j}^{(d)}=\kappa_{j}^{(d)}$, where $1_j=\{\{1,\dots,j\}\}$ is the maximal partition with respect to reverse refinement. Then, \begin{equation}
		\begin{aligned}
			\lim\limits_{d\rightarrow\infty}\kappa_{j}^{(d)}(p_d)-\kappa_{j}(\mu)&=\lim\limits_{d\rightarrow\infty}\kappa_{j}^{(d)}(p_d)\left(1-\frac{Q_{1_{j}}(d)}{d^{j}}\right)+m_{j}(\mu_{p_d})-m_{j}(\mu)\\
			&\quad+\sum_{\sigma\in\mathcal{NC}(j),\sigma\neq 1_{j}}\kappa_{\sigma}(\mu)-\frac{Q_{\sigma(d)}}{d^{j+1-|\sigma|}}\kappa_{\sigma}^{(d)}(p_{d}).
		\end{aligned}
	\end{equation} The proof of $(2)\Rightarrow(1)$ then proceeds by a simple induction after noting $\kappa_{1}^{(d)}(p_d)=m_{1}(\mu_{p_d})$ and $\kappa_{1}(\mu)=m_{1}(\mu)$. 
		
		The second assertion on the rectangular convolution follows from \cite[Theorem 5.1]{Cuenca24} and is implicit in the work of \cite{Gribinski2024}.
	\end{proof} 

	Before proving our main results we sketch proofs of the smaller Lemmas \ref{lem:LP1 test} and \ref{lem:rect matrix compression}. 
	\begin{proof}[Proof of Lemma \ref{lem:rect matrix compression}]
		We will need to define the notion of minor orthogonality. For this we use the notation of \cite{Marcus-Spielman-Srivastava2022}. For a set $S$ of size $n$, we write $\binom{S}{k}$ to be the collection of size $k$ subsets of $S$. For a $m\times n$ matrix $A$, and subset $S\subset [m],T\subset [n]$ with $|S|=|T|$ we define the $(S,T)$-minor of $A$ by \begin{equation}\label{eq:minor def}
			[A]_{S,T}:=\det\left(\{a_{ij}\}_{i\in S,j\in T}\right).
		\end{equation} For integers $m,n,r,k$, $m\times n$ matrix $A$ and $n\times r$ matrix $B$ we have from \cite[Theorem 2.2]{Marcus-Spielman-Srivastava2022} that \begin{equation}\label{eq:minor product}
		[AB]_{S,T}=\sum_{U\in\binom{[n]}{k}} [A]_{S,U}[B]_{U,T},
	\end{equation} for any sets $S\in\binom{[m]}{k}$ and $T\in\binom{[p]}{k}$. We now note that the matrices $U$ and $O$ are minor-orthogonal, and hence for any subsets $S,T,Q,P$ with $|S|=|T|=k$ and $|Q|=|P|=j$ \begin{equation}\label{eq:minor orth}
	\E_{U}\left([U]_{S,T}[U^*]_{Q,P} \right)=\frac{1}{\binom{d}{k}}\delta_{S=P}\delta_{T=Q}.
\end{equation} An analogous equality holds for $O$, with $d$ replaced by $d+n$. We next note that if we denote the $(-1)^{k}x^{d-k}$ coefficient of the characteristic polynomial of a $d\times d$ square matrix $B$ by $\sigma_{k}(B)$, then by \cite[See (14)]{Marcus-Spielman-Srivastava2022}\begin{equation}\label{eq:coeff form}
\sigma_{k}(B)=\sum_{|S|=k}[B]_{S,S}.
\end{equation} Let $C=UAO$. Combining \eqref{eq:minor product}, \eqref{eq:minor orth}, and \eqref{eq:coeff form} in a way nearly identical to the proof of Lemma 1.17 in \cite{Marcus-Spielman-Srivastava2022} we see that for any $1\leq k \leq \ell$ \begin{equation}
\begin{aligned}
	\E_{U,O}\sigma_{k}(CC^*)&=\frac{\binom{\ell}{k}}{\binom{d}{k}}\E_{O}\sum_{|S|=k}[AOO^*A^*]\\
	&=\frac{\binom{\ell}{k}}{\binom{d}{k}}\frac{\binom{\ell +n}{k}}{\binom{d+n}{k}}\sum_{|S|=|T|=k}[A]_{S,T}[A^*]_{T,S}\\
	&=\frac{\binom{\ell}{k}}{\binom{d}{k}}\frac{\binom{\ell +n}{k}}{\binom{d+n}{k}}\sigma_{k}(AA^*).
\end{aligned}
\end{equation} This is exactly the $(-1)^{k}x^{\ell-k}$ coefficient of the right hand side of \eqref{eq:rect comp}, completing the proof.
	\end{proof}
	
	\begin{proof}[Proof of Lemma \ref{lem:LP1 test}]
		Let $a\in\R$. We prove that if $q$ is a polynomial with positive roots, then $(M_n-a^2)q$ is also a polynomial with positive roots. Assume $q$ is degree $d$. Let $b,m$ be such that \begin{equation}
			\begin{aligned}
				b&=\frac{(n+1)-\sqrt{(n+1)^2+4a^2}}{2}\\
				m&=n+\sqrt{(n+1)^2+4a^2}.
			\end{aligned}
		\end{equation} Then, \begin{equation}
			D\left[x^{m+1}D\left(e^{bx}q(x)\right) \right]=x^{m}e^{bx}\left(M_n-a^2\right)q(x).
		\end{equation} By Rolle's theorem the roots of $D e^{bx}q(x)$ interlace with the roots of $q(x)$. It also follows from the fact that $b<0$ that all $d$ roots must be positive. The conclusion then follows by another application of Rolle's theorem after noting that $1+m>0$. 
		
		The proof then follows from a straightforward generalization of the analogous result for functions in the Laguerre--P\'olya class and Jensen polynomials, see for example \cite{OSullivan2021}.
	\end{proof}
	
		\subsection{Heavy-tailed limits of Appell sequences and the proof of Theorem \ref{thm:Cauchy distribution and cosine}}\label{sec:heavy-tailed limit} As the Taylor coefficients of cosine are quite simple there is certainly a shorter proof of Theorem \ref{thm:Cauchy distribution and cosine} than what we present here. However, the proof here is structured in a way to motivate the utility of a true analytic theory of finite free probability.
		
		\begin{proof}[Proof of Theorem \ref{thm:Cauchy distribution and cosine}]
			First, we consider the following added assumption on $f$. Let $n_{+}:(0,\infty)\rightarrow\R$ be the function counting the number of roots in $(0,r)$ and let $n_{-}:(0,\infty)\rightarrow\R$ be the number of roots in $(-r,0)$. We assume there exists $\alpha\in(0,2)$ and slowly varying (see \cite{Resnick87} for background on slow and regular variation) functions $h_{+}$ and $h_{-}$ such that \begin{equation}\label{eq:RV roots}
				\lim\limits_{r\rightarrow\infty}\frac{n_{\pm}(r)}{h_{\pm}(r)r^{\alpha}}=1,\text{ and }\lim\limits_{r\rightarrow\infty}\frac{n_{+}(r)}{n_{+}(r)+n_{-}(r)}=\theta\in [0,1].
			\end{equation} We note that \eqref{eq:RV roots} is also the condition under which a conjecture of Farmer and Rhoades \cite{Farmer-Rhoades2005} on universality of root spacing under repeated differentiation was made. See also \cite{Campbell-ORourke-Renfrew2024even} for the resolution of this conjecture for even/odd functions. The assumptions \eqref{eq:RV roots} tells us how the measure in Theorem \ref{thm:Appell point proc} varies at $0$. At this point we assume, for simplicity, that $h(r)=h_{+}(r)+h_{-}(r)$ can be chosen to be a constant. We consider the rescaling of the Appell polynomials ${C}_{d}(z)=\mathcal{D}_{d^{\rho}}A_{d,f}$ where $\rho=\frac{1}{\alpha}-1$. We let $G_{d}$ denote the Cauchy transform of ${C}_d$, and note that \begin{equation*}
				G_{d}(z)=\frac{1}{d}\frac{C_{d}'(z)}{C_{d}(z)}.
			\end{equation*} For any $f$ one can show that \begin{equation}\label{eq:Cauchy as finite Voic transform}
				G_{d}(z)=\frac{1}{z}-\frac{1}{z^2}\frac{d^{\rho}}{d}\frac{J_{d,f}'\left(\frac{d^{1/\alpha}}{zd}\right)}{J_{d,f}\left(\frac{d^{1/\alpha}}{zd}\right)}.
			\end{equation} We now specialize to the case where $f(z)=\cos(z)$ and $\rho=0$. The Jensen polynomials of cosine have the simple formula \begin{equation}
				J_{d,\cos}(z)=\frac{(1+iz)^{d}+(1-iz)^{d}}{2}.
			\end{equation} Plugging this into \eqref{eq:Cauchy as finite Voic transform} and taking the limit we see that for $\Im(z)>0$, \begin{equation}
				\begin{aligned}
					\lim\limits_{d\rightarrow\infty}G_{d}(z)&=\lim\limits_{d\rightarrow\infty}\frac{1}{z}-\frac{i}{z^2}\left[\frac{\left(1+\frac{i}{z}\right)^{d-1}-\left(1-\frac{i}{z}\right)^{d-1}}{\left(1+\frac{i}{z}\right)^d+\left(1-\frac{i}{z}\right)^{d} } \right]\\
					&=\frac{1}{z}-\frac{i}{z^2}\frac{z}{z+i}\\
					&=\frac{1}{z+i},
				\end{aligned}
			\end{equation} which is exactly the Cauchy transform of the Cauchy distribution. This completes the proof.
		\end{proof}
		
		We now discuss what one may try in the above proof if we did not have such a simple expression for the Jensen polynomials. Let \begin{equation*}
		    \varphi_{d,f}(z)=-\frac{d^{\rho}}{d}\frac{J_{d,f}'\left(\frac{d^{1/\alpha}}{zd}\right)}{J_{d,f}\left(\frac{d^{1/\alpha}}{zd}\right)},
		\end{equation*} and note that by \eqref{eq:Cauchy as finite Voic transform} that \begin{equation*}
		    \varphi_{d,f}(z)=z^2\left[G_{d}(z)-\frac{1}{z}\right].
		\end{equation*} The Voiculescu transform of a probability measure $\mu$ is the function $\varphi_{\mu}(z)=R_{\mu}\left(\frac{1}{z}\right)$. For any probability measure $\mu$ one has, see \cite[Proposition 2.5]{Bercovici-Pata1999}, the relation \begin{equation*}
			\varphi_{\mu}(z)=z^{2}\left[G_{\mu}(z)-\frac{1}{z} \right](1+o(1)),
		\end{equation*} as $z\rightarrow\infty$ non-tangentially. Thus for large $z$, $\varphi_{d,f}$ is a good approximation of the true Voiculescu transform of $\mu_{C_{d}}$. Define the function $\tilde{R}_{d,f}$ by \begin{equation*}
			\tilde{R}_{d,f}(z)=\phi_{d,f}\left(\frac{1}{z}\right).
		\end{equation*}  We also know that $J_{d,f}(w/d)$ converges to $f(w)$ uniformly on compact subsets. The question then becomes, how good of an approximation is $\tilde{R}_{d,f}$ of the un-truncated version of the finite free $R$-transform $\hat{R}_{d}(z)=-d^{\rho}\frac{f'(d^{1/\alpha}z)}{f(d^{1/\alpha}z)}$? Rewriting this we see that we are asking whether \begin{equation}\label{eq:log der estimate}
			\frac{1}{d}\frac{\dd}{\dd z}\left[\log J_{d,f}\left(\frac{d^{1/\alpha}z}{d}\right)-\log f\left(d^{1/\alpha}z\right)\right]\rightarrow0
		\end{equation} on some set with an accumulation point in the upper-half plane. This is the key technical step missing from extending Theorem \ref{thm:Cauchy distribution and cosine} to all $f$ satisfying assumption \eqref{eq:RV roots}. In general, the factor of $d^{1/\alpha}$ ruins the approximation of $f$ by $J_{d.f}$, however we are only asking that the logarithmic derivatives are within $o(d)$ of each other. 
		
		We now discuss how one would extend Theorem \ref{thm:Cauchy distribution and cosine} to $f$ satisfying \eqref{eq:RV roots} to recover other $\boxplus$-stable distributions if one could prove \eqref{eq:log der estimate}. As can be seen from the Bercovici--Pata bijection and a result of Maller and Mason \cite{Maller-Mason2008} on the small time limit of L\'evy processes, the only possible small time limit of the free L\'evy process associated to the limiting measure in Theorem \ref{thm:Appell and free ID} is a stable distribution,  see also \cite{Arizmendi-Hasebe2018} for a discussion of the free case. We give an ad-hoc proof below which uses \eqref{eq:RV roots}. This assumption is not only sufficient for computing the following limit, but by \cite[Theorem 2.3]{Maller-Mason2008} and the Bercovici--Pata bijection it is in fact necessary. Let $h(r)=h_{+}(r)+h_{-}(r)$. Let $u>0$ and consider the rescaling of the reciprocal root point  measure of $f$ given by \begin{equation}
			\Pi_{f,u}= u^{\alpha}h(u^{-1})^{-1}\sum_{j=1}^{\infty} \delta_{\frac{1}{ux_{j}}}.
		\end{equation} Then, for any $a>0$ \begin{equation}\label{eq:f point conv}
			\lim\limits_{u\rightarrow0^{+}} \Pi_{f,u}((a,\infty))=\theta\int_{a}^{\infty}\alpha r^{-\alpha-1}\dd r. 
		\end{equation} An analogous limit holds for $(-\infty, -a)$. Thus, if $\overline{\R}$ is the one point compactification of $\R$, then as Radon measures on $\overline{\R}\setminus\{0\}$ endowed with the subspace topology $\Pi_{f,u}$ converges to $\dd m_{\alpha}=\theta r^{-\alpha-1}\indicator{r>0}\dd r+(1-\theta)|r|^{-\alpha-1}\indicator{r<0}\dd r$, which is the L\'evy measure of an $\alpha$-stable distribution, and by the Bercovici--Pata bijection \cite{Bercovici-Pata1999} also of a free $\alpha$-stable distribution. We assume now that $\sigma^2=0$ and for simplicity that $h$ can be chosen to be a constant. Then, \begin{equation}\label{eq:heavy limit}
			\hat{R}_{d}(s)=d^{\rho}c+d^{-1}\sum_{j=1}^{\infty}\frac{1}{\frac{x_j}{d^{1/\alpha}}-s }-\frac{d^{1/\alpha}}{x_{j}}.
		\end{equation} We now want to set $u=d^{-1/\alpha}$ and use \eqref{eq:f point conv} to compute the limit. We will need to break this limit up into different $\alpha$ regimes. If $\alpha\in (1,2)$ then we choose $c=0$ and get \begin{equation}
			\lim\limits_{d\rightarrow\infty}\hat{R}_{d}(s)=\int_{\R}\frac{sr^2}{1-sr}\dd m_{\alpha}(r),
		\end{equation} uniformly on compact subsets of $\C^{+}$, which (after some algebra) can be seen to be the $R$-transform of a free $\alpha$-stable distribution. If $\alpha\in (0,1)$ we choose $c=\sum_{j=1}^{\infty}\frac{1}{x_{j}}$ and recover the limit \begin{equation}
			\lim\limits_{d\rightarrow\infty}\hat{R}_{d}(s)=\int_{\R}\frac{r}{1-sr}\dd m_{\alpha}(r).
		\end{equation} Similarly for $\alpha=1$ a judicious choice of $c$ and a more careful limit leads to the $R$-transform of a $\boxplus$-stable distribution. Hence, if one could show \eqref{eq:log der estimate} then we would have that $G_{d}$ converges to some $G$ where one could show, using the arguments of \cite[Seciton 5]{Bercovici-Pata1999} that there exists $v$ which is $(-1-\alpha)$-regularly varying such that  \begin{equation*}
			G(iy)-\frac{1}{iy}=\left(i+\theta\tan\frac{\pi\alpha}{2}\right)v(y)(1+o(1)),
		\end{equation*} as $y\rightarrow\infty$. Using \cite[Propositions 5.9 and 5.10]{Bercovici-Pata1999} this would characterize the limit as $\boxplus$-stable.

		\subsection{Proofs of Theorems \ref{thm:Appell and free ID} and \ref{thm:L-A and Rect ID}} In this section we will not rely on the full strength of Lemma \ref{lem:R-transform cumulant equivalence}. Instead we will rely only on the direction which does not require a uniform bound on the supports of the measures. Note that $\mu_{f}$ and $\tilde{\mu}_{g}$ both have moments of all orders. Furthermore, for any $j$ the (rectangular) finite free cumulants of the polynomials $\hat{A}_{d}$ and $\hat{L}_{d}$ are fixed in $d$ by design. Thus, from the moment-cumulant formulas \eqref{eq:moment cumulant formula} for any fixed $j$, the $j$-th moment of $\mu_{\hat{A}_{d}}$ or $\mu_{\hat{L}_d}$ is uniformly bounded in $d$. Thus, though we do not necessarily have a uniform bound on the largest roots of $\hat{A}_{d}$ or $\hat{L}_{d}$ we will only use that the moments converge in $d$ and the limiting measures $\mu_{f}$ and $\tilde{\mu}_{g}$ are determined by their moments.  This is summarized in the following lemma. The proof of the first assertion is identical to the $(1)\Rightarrow(2)$ direction of the proof of Lemma \ref{lem:R-transform cumulant equivalence}. The second assertion follows from \cite[Theorem 5.1]{Cuenca24}. \begin{lemma}\label{lem:moment method lemma}
			Let $\mu$ be a probability measure determined by its moments and let $\{p_d\}$ be a sequence of real rooted polynomials indexed by their degree such that the finite free $R$-transforms $R_{p_d}^{d}$ converge coefficient-wise as $d\rightarrow\infty$ to the power series expansion of $R_{\mu}$ around zero. Then, $\mu_{p_d}$ converge weakly to $\mu$.
			
		Additionally, if all the roots of $p_{d}$ are non-negative for every $d$, $\mu$ is supported on $[0,\infty)$, and the $(d,n)$-rectangular finite free $R$-transforms $R_{p_d}^{d,n}$ converge coefficient-wise as $d\rightarrow\infty$, $\frac{d}{d+n}\rightarrow\lambda\in (0,1]$ to the power series expansion of $C_{\sqrt{\mu}}^{\lambda}$ around zero, then $\mu_{p_d}$ converges weakly to $\mu$ as $d\rightarrow\infty$.
		\end{lemma} 
	\begin{proof}[Proof of Theorem \ref{thm:Appell and free ID}]
		One can see  from  \eqref{eq:finite Fourier transforms}, \eqref{eq:finite free R}, \eqref{eq:exchanging log der and trunc},  \eqref{eq:normalized Appell}, and some algebra that \begin{equation}\label{eq:Appell R computation}
			\begin{aligned}
				R_{\hat{A}_{d}}^{d}(s)&=-\frac{1}{d}\frac{\dd}{\dd s}\log f(s)^d&\mod[s^{d}]\\
				&=-\frac{f'(s)}{f(s)}&\mod[s^{d}]\\
				&=-c+\sigma^2s+\sum_{j=1}^\infty\frac{1}{x_{j}-s}-\frac{1}{x_j}&\mod[s^{d}]\\
				&=-c+\sigma^2s+\sum_{j=1}^{\infty}\sum_{k=1}^{\infty}\frac{s^{k}}{x_{j}^{k+1}}&\mod[s^{d}]\\
				&=-c+\left(\sigma^2+\sum_{j=1}^{\infty}\frac{1}{x_{j}^2}\right)s+\sum_{k=2}^{\infty}\left(\sum_{j=1}^{\infty}\frac{1}{x_{j}^{k+1}}\right)s^{k}&\mod[s^{d}].\\
			\end{aligned}
		\end{equation}  We now note that $\nu_{f}$ is compactly supported, and we consider $|s|\in\C$ such that $s<\inf_{j\in\N}{|x_{j}|}$. Then, \begin{equation}\label{eq:analytic R transform expansion}
		\begin{aligned}
			R_{\mu_{f}}(s)&=\gamma+\sigma^2s+\int_{\R}\frac{s+t}{1-st}\frac{t^2}{t^2+1}\dd\nu_{f}(t)\\
			&=\gamma+\sigma^2s+\int_{\R}\left[\frac{st^2}{1-st}+\frac{t^3}{t^2+1}\right]\dd\nu_{f}(t)\\
			&=\gamma+\sigma^2s+\sum_{j=1}^{\infty}\frac{s}{x_j(x_j-s)}+\frac{1}{x_{j}^3+x_{j}}\\
			&=-c+\sigma^2s+\sum_{j=1}^\infty\frac{1}{x_{j}-s}-\frac{1}{x_j}\\
			&=-c+\left(\sigma^2+\sum_{j=1}^{\infty}\frac{1}{x_{j}^2}\right)s+\sum_{k=2}^{\infty}\left(\sum_{j=1}^{\infty}\frac{1}{x_{j}^{k+1}}\right)s^{k}.
		\end{aligned}
	\end{equation} 	
	We then take the $d\rightarrow\infty$ limit of \eqref{eq:Appell R computation} to recover \eqref{eq:analytic R transform expansion} on some sufficiently small neighborhood of $s=0$. The proof is then complete by Lemma \ref{lem:moment method lemma}. 
	\end{proof}

	The proof of Theorem \ref{thm:L-A and Rect ID} follows in a nearly identical manner as the proof of Theorem \ref{thm:Appell and free ID}. \begin{proof}[Proof of Theorem \ref{thm:L-A and Rect ID}]
		Using Definition \ref{def:finite R transforms},  \eqref{eq:exchanging log der and trunc},  and the definition of $\hat{L}_{d}$ we see that as formal power series \begin{equation}\label{eq:rect R comp}
			\begin{aligned}
				R_{\hat{L}_{d}}^{d,n}(s)&=-s\frac{g'(s)}{g(s)}&\mod[s^{d+1}]\\
				&=s\sigma^{2}+\sum_{j=1}^{\infty}\frac{s}{\alpha_{j}^{2}-s}&\mod[s^{d+1}]\\
				&= \sigma^2s+\sum_{k=1}^{\infty}\left(\sum_{j=1}^{\infty}\frac{1}{\alpha_{j}^{2k}}\right)s^{k} &\mod[s^{d+1}].
			\end{aligned}
		\end{equation}  As $G_g$ is compactly supported we take $s\in\C$ such that $|s|<\inf_{j\in\N}(\alpha_{j}^2)$ and expand $C_{\tilde{\mu}}^{\lambda}$ to see that \begin{equation}
		\begin{aligned}
			C_{\tilde{\mu}}^{\lambda}(s)&=s\int_{\R}\frac{t^2+1}{1-st^2}\dd G_g(t)\\
			&= s\sigma^{2}+s\sum_{j=1}^{\infty}\frac{\alpha_{j}^2+1}{\alpha_{j}^{2}+1}\frac{\frac{1}{\alpha_{j}^2}}{1-\frac{s}{\alpha_{j}^{2}}}\\
			&=s\sigma^{2}+\sum_{j=1}^{\infty}\frac{s}{\alpha_{j}^{2}-s}\\
			&= \sigma^2s+\sum_{k=1}^{\infty}\left(\sum_{j=1}^{\infty}\frac{1}{\alpha_{j}^{2k}}\right)s^{k}.\\
		\end{aligned}
	\end{equation} The proof is then complete by taking $d\rightarrow\infty$ in \eqref{eq:rect R comp} and applying Lemma 4.2. 
	\end{proof}

	\subsection{Proofs of Theorems \ref{thm:Appell point proc}, \ref{thm:Appell domain}, \ref{thm:L-Appell point process conv}, and \ref{thm: rect finite ID domain of attraction}} We begin with Theorems \ref{thm:Appell point proc} and \ref{thm:L-Appell point process conv}. \begin{proof}[Proof of Theorem \ref{thm:Appell point proc} ] 
		First, fix $\eps>0$. Note that $\tilde{A}_{d}(z)=J_{d,f}\left(\frac{1}{dz}\right)$, hence the roots of $\tilde{A}_{d}$ are the reciprocals of the roots of $J_{d,f}$ scaled by $d$. Recall that \begin{equation}
			\lim\limits_{d\rightarrow\infty}J_{d,f}\left(\frac{w}{d}\right)=f(w),
		\end{equation} uniformly on compact subsets. Hence, the for any set $B\subset\R\setminus(-\eps,\eps)$ the multi-set of roots of $\tilde{A}_{d}$ in $B$ converges to the reciprocal of the multi-set  of roots of $f$ in $\{z: z^{-1}\in B\}$. Let $\dd\rho_{d}=dt^2\dd\mu_{\tilde{A}_{d}}(t)$. As we have just seen $\rho_{d}(B)$ converges to $G_{f}(B)$ for any $B$ bounded away from $0$. This would complete the proof, except that $G_f$ has a point mass at the origin, and thus we have not yet proven convergence for all continuity sets of $G_f$. 
	
	It also follows from the Newton identities that \begin{equation}
		\rho_{d}(\R)=dm_{2}(\tilde{A}_d)=\gamma_{1}^2-\frac{d(d-1)}{d^2}\gamma_{2},
	\end{equation} and hence the sequence is tight.  We note that again by Newton's identities that \begin{equation}
	\kappa_{1}^{d}(A_d)=m_{1}(\tilde{A}_{d})=\frac{1}{d}\gamma_{1},
\end{equation} thus $\frac{m_2(\tilde{A}_2)}{\kappa_{2}^{d}(\tilde{A}_d)}\rightarrow1$ as $d\rightarrow\infty$. Thus, by examining the coefficients of the $R$-transform we see that $\rho_{d}(\R)\rightarrow\sigma^{2}+\sum_{j=1}^{\infty}\frac{1}{x_{j}^2}=G_{f}(\R)$. Thus, for any continuity set $B$ of $G_{g}$, $\rho_{d}(B)\rightarrow G_{f}(B)$. 
\end{proof} The proof of Theorem \ref{thm:L-Appell point process conv} follows in a completely analogous way to that of Theorem \ref{thm:Appell point proc}. We now move on to Theorems \ref{thm:Appell domain} and \ref{thm: rect finite ID domain of attraction}. 

\begin{proof}[Proof of Theorem \ref{thm:Appell domain}]
	We begin with a bound on the roots of $p_d$. Let $y_d=\max_{p_{d}(x)=0}\{|x|\}$. Given that $m_{2}(p_2)=O(d)$ we can see immediately that $y_{d}=O(d)$. Let $p_{d,\ell}=\frac{\ell!}{d!}D^{d-\ell}p_{d}$. From \cite[Proposition 3.4]{Arizmendi-Fujie-Perales-Ueda2024} we see that \begin{equation}
		\kappa_{j}^{\ell}(p_{d,\ell})=\left(\frac{\ell}{d}\right)^{j-1}\kappa_{j}^{d}(p_d),
	\end{equation} for any $j\in[\ell]$. We note by convexity and our assumption on the second moments of $p_d$ that $m_{j}(p_d)=O(d^{j-1})$ for any $j\geq 2$. Thus, we see from the moment-cumulant formula that \begin{equation}
	\kappa_{j}^{d}(p_d)=\frac{d^{j}}{(d)_{j}}m_{j}(p_d)+o\left(d^{j-1}\right).
	\end{equation} Let $\dd\rho_{d}(t)=dt^{2}\dd\mu_{\mathcal{D}_{d^{-1}}p_d}(t)$ and note that for any $j\geq 2$ \begin{equation}
	\frac{1}{d^{j-1}}m_{j}(p_d)=\int_{\R}t^{j-2}\dd\rho_{d}(t). 
\end{equation} Let $C>0$ be such that $y_{d}<Cd$ and let $h_j(t)$ be continuous such that $h_{j}(t)=t^{j-2}$ for $|t|\leq C$ and $h_{j}(t)=0$ for $|t|\geq C+1$. Then, for any $j\geq 2$ \begin{equation}\label{eq:Domain of attract cumulant comp}
\begin{aligned}
	\lim\limits_{d\rightarrow\infty}\kappa_{j}^{\ell}(p_{d,\ell})&=\lim\limits_{d\rightarrow\infty}\left(\frac{\ell}{d}\right)^{j-1}\kappa_{j}^{d}(p_d)\\
	&=\lim\limits_{d\rightarrow\infty}\ell^{j-1}\frac{1}{d^{j-1}}m_{j}(p_d)\\
	&=\ell^{j-1}\lim\limits_{d\rightarrow\infty}\int_{\R}t^{j-2}\dd\rho_{d}(t)\\
	&=\ell^{j-1}\lim\limits_{d\rightarrow\infty}\int_{\R}h_{j}(t)\dd\rho_{d}(t)\\
	&=\ell^{j-1}\int_{\R}t^{j-2}\dd G_{f}(t).
\end{aligned}
\end{equation} As can be seen from a series expansion of $R_{f(D)x^{\ell}}^{\ell}(s)=-\frac{f'(\ell s)}{f(\ell s)}\mod[s^{\ell}]$  using \eqref{eq:exchanging log der and trunc}, the last line of \eqref{eq:Domain of attract cumulant comp} is exactly the $j$-th finite free cumulant of $f(D)x^{\ell}$ for $j\geq 2$.  The proof is then completed by noting $\kappa_{1}^{\ell}(p_{d,\ell})=\kappa_{1}^{d}(p_{d})$ for all $d\geq 1$.
\end{proof} The proof of Theorem \ref{thm: rect finite ID domain of attraction} follows in a nearly identical manner. We explain the minor changes below.

\begin{proof}[Proof of Theorem \ref{thm: rect finite ID domain of attraction}]
	The modifications of the proof of Theorem \ref{thm:Appell domain} are as follows. First, replace \cite[Proposition 3.4]{Arizmendi-Fujie-Perales-Ueda2024} with Lemma \ref{thm:Cumulant theorem}. Then, from the moment assumptions and convexity $m_{j}(p)=O(d^{2j-1})$ for any $j\geq 2$. One can combine \cite[Theorem 3.3 and Definition 4.3]{Cuenca24} with \eqref{eq:scaled finite rect cumulants} to see that \begin{equation}
		\kappa_{2j}^{n,d}[p]=\frac{(n+d)^{j}d^{j-1}}{(n+d)_{j}(d-1)_{j-1}}m_{j}(p)+o(d^{2j-1}).
	\end{equation} The conclusion of the proof is then identical.
\end{proof} 
	
	\subsection{Proofs of Lemma \ref{lem:R transform lemma} and Lemma \ref{thm:Cumulant theorem}} These lemmas follow from straightforward computations, but we present them for completeness here. 
	
	\begin{proof}[Proof of Lemma \ref{lem:R transform lemma}]
		Let $P_{d}$ be  a formal power series  such that $P_{d}\left(\frac{D}{d}\right)x^{d}=p(x)$. Note that $P_{d}\left(\frac{D}{d} \right)x^{d-j}=p_{j}(x)$. Let $P_{d-j}(s)=P_{d}\left(\frac{d-j}{d}s\right)$. Then,  $R_{p}^d=-\frac{1}{d}\frac{P_{d}'(s)}{P_{d}(s)}\mod[s^{d}]$, and \begin{equation}
			\begin{aligned}
					R_{p_{j}}^{d-j}(s)&=-\frac{1}{d-j}\frac{P_{d-j}'\left(s\right)}{P_{d-j}(s) }&\mod[s^{d-j}]\\
				&=-\frac{1}{d}\frac{P_{d}'\left(\frac{d-j}{d}s \right)}{P_{d}\left(\frac{d-j}{d}s \right)}&\mod[s^{d-j}]\\
				&=R_{p}^{d}\left(\frac{d-j}{d}s\right)&\mod[s^{d-j}].
			\end{aligned} 
		\end{equation} For the $(d,n)$-rectangular finite free $R$-transform we let $P_{d,n}(s)$  be  a formal power series  such that $P_{d,n}\left(\frac{M_n}{d(d+n)} \right)x^{d}=p(x)$ and let $P_{d-j,n}(s)=P_{d,n}\left(\frac{(d-j)(n+d-j)}{d(n+d)}s \right)$. Then, \begin{equation}
		\begin{aligned}
			R_{\hat{p}_{j}}^{d-j,n}(s)&=-\frac{s}{d-j}\frac{P_{d-j,n}'(s) }{P_{d,n}(s)}&\mod[s^{d-j+1}]\\
									  &=-\frac{s}{d}\frac{n+d-j}{n+d}\frac{P_{d,n}'\left(\frac{(d-j)(n+d-j)}{d(n+d)}s \right)}{P_{d,n}\left(\frac{(d-j)(n+d-j)}{d(n+d)}s \right)}&\mod[s^{d-j+1}]\\
									  &=\frac{d}{d-j}R_{p}^{d,n}\left(\frac{(d-j)(n+d-j)}{d(n+d)}s\right)&\mod[s^{d-j+1}].
		\end{aligned}
	\end{equation} Finally, \eqref{eq:rect to square} follows from a simple calculation of the coefficients of $P_{d,n}$ and $P$ as defined in Definition \ref{def:finite R transforms}. This completes the proof.
	\end{proof}
	\begin{proof}[Proof of Lemma \ref{thm:Cumulant theorem}]
		Let $p(x)=x^{d} +\sum_{k=1}^{d}(-1)^{k}a_{k}x^{d-k}$ and \begin{equation}
			p_{j}(x)=\frac{1}{(d)_{j}(n+d)_j}M_{n}^{j}p(x)=x^{d-j}+\sum_{k=1}^{d-j}(-1)^{k}a_{k}^{(j)}x^{d-j-k}.
		\end{equation} It is straightforward to see that \begin{equation}
		a_{k}^{(j)}=\frac{(d-k)_{j}(n+d-k)_{j}}{(d)_j(n+d)_j}a_{k}.
	\end{equation} One can then check that \begin{equation}\label{eq:rect coeff formula}
	\frac{a_{k}^{(j)}}{(d-j)_{k}(n+d-j)_{k}}=\frac{a_k}{(d)_k(n+d)_{k}}.
\end{equation} Plugging this into \eqref{eq:unscaled finite rect cumulants} completes the proof of \eqref{eq:cumulant formula Mn}. The proof of \eqref{eq:scaled cumulant formula Mn} then follows immediately from \eqref{eq:cumulant formula Mn}. 
	\end{proof}

	\subsection{Proof of Theorem \ref{thm:Limit theorem}} Theorem \ref{thm:Limit theorem} follows from computing the limits of the finite $R$-transforms at our various scales, as captured in the next proposition. \begin{proposition}\label{prop:R-transform limits}
		Let $\{p_{d}\}_{d=1}^\infty$ be a sequence of polynomials with only non-negative roots satisfying Assumption \ref{assumption:limiting measure} with some limiting root measure $\mu_{p}$. For any $j\in\{1,\dots,d\}$\begin{equation}
			p_{j,d}(x)=\frac{1}{(d)_{j}(n+d)_{j}}M_{n}^{j}p_{d}(x).
		\end{equation} 
	
	Assume $\ell=d-j$ is fixed:\begin{enumerate}
		\item If $n>-1$ is fixed and the first moment of $\mu_{p}$ is $1$, then \begin{equation}\label{eq:Lag R limit}
			\lim\limits_{d\rightarrow\infty}R_{\mathcal{D}_{{d+n}}p_{j,d}}^{\ell,n}(s)=(n+\ell)s.
		\end{equation} 
		\item If $n\rightarrow\infty$ as $d\rightarrow\infty$, and the first moment of $\mu_{p}$ is $1$, then \begin{equation}\label{eq:LLN R limit}
			\lim\limits_{d\rightarrow\infty}R_{\mathcal{D}_{\frac{d+n}{n}}p_{j,d}}^{\ell}(s)=1.
		\end{equation} 
	\end{enumerate} 

Next assume $\ell=d-j\rightarrow\infty$ and $\ell=o(d)$:\begin{enumerate}
\item If $-1<n=o(\ell)$ and the first moment of $\mu_{p}$ is $1$, then  \begin{equation}\label{eq:MP1 R limit}
	\lim\limits_{d\rightarrow\infty}R_{\mathcal{D}_{\frac{n+d}{n+\ell}}p_{j,d}}^{\ell,n}(s)=s.
\end{equation} 
\item If $n\sim \alpha\ell$ for some $\alpha\in(0,\infty)$, and the first moment of $\mu_{p}$ is $1$, then \begin{equation}\label{eq:MPalpha R limit}
	\lim\limits_{d\rightarrow\infty}R_{\mathcal{D}_{\frac{n+d}{n+\ell}}p_{j,d}}^{\ell,n}(s)=s.
\end{equation}
\end{enumerate}

Finally, assume $d-j\sim t d$ for $t\in (0,1)$:\begin{enumerate}
\item If $-1<n=o(d)$, then \begin{equation}\label{eq:Fract power R limit}
	\lim\limits_{d\rightarrow\infty}R_{p_{j,d}}^{d-j,n}(s)=\frac{1}{t}C_{\sqrt{\mu_{p}}}^{{1}}\left(t^2s \right),
\end{equation} uniformly for $s$ sufficiently small.
\item If $n\sim\alpha d$, then \begin{equation}\label{eq:mixed ratio R limit}
	R_{p_{j,d}}^{d-j,n}(s)=\frac{1}{t}C_{\alpha}\left(t\frac{\alpha+t}{\alpha+1}s\right).
\end{equation}
\end{enumerate}
	\end{proposition} \begin{proof}[Proof of Theorem \ref{thm:Limit theorem}]
	Theorem \ref{thm:Limit theorem} follows from Proposition \ref{prop:R-transform limits} and Lemma \ref{lem:R-transform cumulant equivalence}. When $\ell=d-j$ and $n$ are fixed the finite rectangular $R$-transform $R^{\ell,n}(s)=(n+\ell)s$ is exactly the finite rectangular $R$-transform of $L^{(n)}_{\ell}$. Similarly when $n\rightarrow\infty$ we recover the finite free $R$-transform of $(z-1)^{\ell}$. 
	
	When $\ell=o(d)$ and $n=o(\ell)$ the remaining degrees are in the regime where $\frac{\ell}{\ell+n}\rightarrow1$, hence the $R^{\ell,n}$ rectangular $R$-transform converges to the $\boxplus_{1}$ rectangular $R$-transform.  The limit in \eqref{eq:MP1 R limit} is exactly the rectangular $R$ transform of the $\boxplus_{1}$-Gaussian with variance $1$, which is the symmetric square root of $\nu_{1}$. Similarly when $n\sim \alpha\ell$, we see that $\frac{\ell}{\ell+n}\rightarrow\frac{1}{1+\alpha}$ and \eqref{eq:MPalpha R limit} is the rectangular $R$-transform of the $\boxplus_{\frac{1}{1+\alpha}}$ Gaussian with variance $1$, which is the symmetric square root of $\nu_\alpha$.
	
	When $d-j\sim td$ for $t\in (0,1)$ and $-1<n=o(d)$ we note that $\frac{d-j}{d-j+n}\rightarrow1$ and hence $R^{d-j,n}$ converges to the $\boxplus_{1}$ rectangular $R$-transform. It is then straightforward to see that \eqref{eq:Fract power R limit} is the rectangular $R$-transform of a convolution power and dilation of $\sqrt{\mu_p}$. Finally, \eqref{eq:mixed ratio limit} follows immediately from \eqref{eq:mixed ratio R limit}. 
\end{proof}

\begin{proof}[Proof of Proposition \ref{prop:R-transform limits}]
	 From Lemma \ref{lem:R transform lemma} \begin{equation}\label{eq:R transform relation}
	 	R_{p_{j,d}}^{d-j,n}(s)=\frac{d}{d-j}R_{p_d}^{d,n}\left(\frac{(d-j)(n+d-j)}{d(n+d)}s\right)\mod\left[s^{d-j+1}\right].
	 \end{equation} We note as pointed\footnote{The conclusions of \cite[Lemmas 7.2 and 7.3]{Gribinski2024} are correct, however there is a small mistake in the definition of the dilation operator there. The operator, as defined, divides the roots by $a$ while the conclusions are drawn for the operator which multiples the roots by $a$.} out in \cite[Lemma 7.3]{Gribinski2024} that for any polynomial with non-negative roots $R_{\mathcal{D}_{a}p }^{d,n}(s)=R_{p}^{d,n}(as)$. We first consider \eqref{eq:Lag R limit}. Scaling \eqref{eq:R transform relation} we see that \begin{equation}\label{eq:R computation example}
	\begin{aligned}
			 R_{\mathcal{D}_{{d+n}}p_{j,d}}^{\ell,n}(s)&=R_{p_{j,d}}^{\ell,n}\left((n+d)s \right)&\mod\left[s^{\ell+1} \right]\\
			 &=\frac{d}{\ell}R_{p_d}^{d,n}\left(\frac{\ell(n+\ell)}{d}s \right)&\mod\left[s^{\ell+1}\right]\\
			 &=(n+\ell)s+O\left(\frac{1}{d}\right)&\mod\left[s^{\ell+1}\right].
	\end{aligned}
 \end{equation} Taking the limit completes the proof of \eqref{eq:Lag R limit}. For \eqref{eq:LLN R limit} we note that the same argument shows \begin{equation}
 \lim\limits_{d\rightarrow\infty}R_{\mathcal{D}_{\frac{d+n}{n}}p_{j,d}}^{\ell,n}(s)=s.
\end{equation} Then \eqref{eq:LLN R limit} follows from \eqref{eq:rect to square}. 

For $\ell=d-j\rightarrow\infty$ and $\ell=o(d)$ both \eqref{eq:MP1 R limit} and \eqref{eq:MPalpha R limit} follow immediately from \eqref{eq:R transform relation} and a computation nearly identical to \eqref{eq:R computation example}. 

Finally we consider the cases when $d-j\sim td$ for $t\in (0,1)$. We first consider when $-1<n=o(d)$ and hence $\frac{d}{d+n}\rightarrow 1$. By assumption as formal power series \begin{equation}
	\lim\limits_{d\rightarrow\infty}R_{p_d}^{d,n}(s)=C_{\sqrt{\mu_{p}}}^{\lambda}(s),
\end{equation} for any $n$ such that $\frac{d}{d+n}=\lambda\in (0,1]$. Then, \eqref{eq:Fract power R limit} and \eqref{eq:mixed ratio R limit} both follow immediately from \eqref{eq:R transform relation} under the proper assumptions on $n$. 
\end{proof}


\section*{Acknowledgments}
I would like to thank Theodoros Assiotis, Vadim Gorin, Jonas Jalowy, and Sean O'Rourke for insightful comments on an previous version of the manuscript and for pointing out helpful references. I would also like to thank the anonymous referee for helpful comments.

\section*{Funding}
This work was supported by ERC Advanced Grant ``RMTBeyond'' No. 101020331 and FWF Grant ESP4314224.


\end{document}